\documentclass[11pt]{amsart}
\usepackage{amsmath,amssymb,amsfonts,url,mathptmx,txfonts}
\usepackage{color}
\usepackage{appendix}
\allowdisplaybreaks[4]
\numberwithin{equation}{section}

\newtheorem{lemA}{Lemma}
\newtheorem{thm}{Theorem}[section]

\usepackage[marginal]{footmisc}

\usepackage{amsthm,a0size,bm,indentfirst,wasysym}
\def\Xint#1{\mathchoice
	{\XXint\displaystyle\textstyle{#1}}%
	{\XXint\textstyle\scriptstyle{#1}}%
	{\XXint\scriptstyle\scriptscriptstyle{#1}}%
	{\XXint\scriptscriptstyle\scriptscriptstyle{#1}}%
	\!\int}
\def\XXint#1#2#3{{\setbox0=\hbox{$#1{#2#3}{\int}$ }
		\vcenter{\hbox{$#2#3$ }}\kern-.6\wd0}}

\def\dashint{\Xint-}

\newtheorem{lem}[thm]{Lemma}
\newtheorem{cor}[thm]{Corollary}
\newtheorem{prop}[thm]{Proposition}
\newtheorem{rem}[thm]{Remark}

\usepackage{amsmath}

\allowdisplaybreaks[4]

\begin{document}
\title[4-Order Elliptic Equations]{Classification and a priori estimates for the singular prescribing $Q$-curvature equation on 4-manifold}

\author{Mohameden Ahmedou}

\author{Lina Wu}

\author{Lei Zhang}\footnote{Lei Zhang is partially supported by a Simons Foundation Collaboration Grant (Award number: 584918)}

\address{ Mohameden Ahmedou\\
Department of Mathematics \\
	Giessen University \\
	Arndtstrasse 2, 35392, Giessen, Germany }
\email{Mohameden.Ahmedou@math.uni-giessen.de}

\address{  Lina Wu\\
Department of Mathematics \\
	Beijing Jiaotong University \\
	Beijing, 100044, China }
\email{lnwu@bjtu.edu.cn}

\address{ Lei Zhang\\
Department of Mathematics\\
        University of Florida\\
        1400 Stadium Rd\\
        Gainesville FL 32611}
\email{leizhang@ufl.edu}

\date{\today}

\begin{abstract}
	On  $(M,g)$  a compact riemannian  $4-$manifold  we consider the prescribed  $Q-$curvature equation defined on $M$ with finite singular sources. We first prove a classification theorem for singular Liouville equations defined on $\mathbb R^4$ and perform a concentration compactness analysis. Then  we derive  a quantization result for bubbling solutions and  establish a priori estimate under the assumption that cetain conformal invariant does not take some quantized values. Furthermore we prove a spherical Harnack inequality around   singular sources  provided their   strength is  not an  integer. Such an inequality implies that in this case singular sources are \emph{isolated simple blow up points}.
\end{abstract}

 \maketitle

\begin{center}

\bigskip
\centerline{{\it  AMS subject classification:}  53C21, 35C60, 58J60.}

\end{center}
\tableofcontents

\section{Introduction}

\smallskip

The question of  existence of conformal metrics of constant or more generally prescribed curvature on riemannian manifolds is  a recurrent problem   in differential geometry and geometric analysis. Indeed  a  positive or  a negative answer to such a question has   far reaching consequences on the geometry and topology of the underlying manifold.
The Poincar\'e uniformization's theorem on closed surfaces, the Yamabe problem on riemannian manifolds of dimension $n \geq 3$ and the Nirenberg's problem on standard spheres $\mathbb{S}^n$, just to name a few,  are well known and well   studied mathematical problems. \\
A similar question, which  goes back to Picard \cite{Picard1, Picard2} deals with the existence of conformal  metrics  of constant or prescribed curvature on surfaces with conical singularities. After  the pioneering work of Picard at the beginning of the last century, such a problem has  been systematically investigated by  Berger \cite{Berger}, McOwen\cite{McOwen, McOwen2} and Troyanov\cite{T,T2}. More recently Bartolucci-deMarchis-Malchiodi\cite{BdM} used a Morse theoretical approach to prove further existence and multiplicity results.\\
In this paper, a first part of a series of papers, we address the problem of existence of  conformal  conical  metrics  of  constant, or more generally prescribed  $Q-$curvature on four dimensional riemannian manifolds. In the following we will explain in some details the geometric context of such a problem: \\
Given  $(M,g)$  a compact four-dimensional Riemannian manifold, the  $Q$-curvature and the Paneitz operator are defined respectively, by
\begin{equation}\label{Q-curvature}
Q_g=-\frac{1}{12}\big(\Delta_g R_g-R_g^2+3|\rm{ Ric}_g|^2\big),
\end{equation}
\begin{equation}\label{Peneitz}
P_g \varphi=\Delta_g^2\varphi+{\rm div}_g\Big(\big(\frac{2}{3}R_g g-2{\rm Ric}_g\big)\nabla u\Big),
\end{equation}
where ${\rm Ric}_g$ is the Ricci tensor and $R_g$ is the scalar curvature of $(M,g)$. \\
Similar to second order equations a natural question is the following uniformization statement: \textit{given a four-dimensional Riemannian manifold $(M,g)$, is there a metric $\tilde{g}=e^{2u}g$ in the conformal class of $g$ with constant $Q$-curvature?}

Under the conformal change of metric above, the Paneitz operator is conformally covariant:
\begin{equation}\label{conformal-P}
P_{\tilde{g}}\varphi=e^{-4u}P_g\varphi,
\end{equation}
and the $Q$ curvature of $\tilde{g}$ is given by
\begin{equation}\label{conformal-Q}
P_gu+2Q_g=2Q_{\tilde{g}}e^{4u}.
\end{equation}

From (\ref{conformal-P}) and (\ref{conformal-Q}), the question above is equivalent to the existence of solution to this fourth order equation:
\begin{equation}\label{Q-equation-reg}
P_gu+2Q_g=2\bar{Q}e^{4u},
\end{equation}
where $\bar{Q}$ is a real constant.\\
Integrating with respect to the volume element ${\rm d}V_g$, we can see that
\begin{equation*}
\kappa_P=\int_{M}Q_g{\rm d}V_g
\end{equation*}
is a constant in the conformal class of $g$, here we also point out the Gauss-Bonnet-Chern  formula that links the local curvature to the global topology of $M$ is:
\begin{equation*}
\int_{M}\Big(Q_g+\frac{1}{8}|W_g|^2\Big){\rm d}V_g=4\pi^2\chi_M,
\end{equation*}
where $W_g$ denotes the Weyl's tensor of $(M,g)$ and $\chi_M$ is the Euler characteristic of $M$. From this equality  and the aforementioned conformal covariance  property it is not hard to imagine that $P_g$ and $Q_g$ are related to a number of studies such as Moser-Trudinger type inequatilities, log-determinant formulas and the compactification of locally flat manifolds, see \cite{Beckner,Branson-Chang-Yang,Branson-Oersted,Chang-Qing-Yang-Invent,Chang-Qing-Yang-Duke,Chang-Yang}. In many of these studies the kernel of $P_g$ is usually assumed to consist only of constants:
\begin{equation*}\label{P-assumption}
{\rm Ker}\,(P_g)=\{constants\}. \leqno(P)
\end{equation*}

In this paper consider the following   prescribed $Q$-curvature equation involving singular sources
\begin{equation}\label{Q-singular}
P_gu+2Q_g=2h e^{4u}-8\pi^2\sum_{j=1}^{N}\gamma_j\Big(\delta_{q_j}-\frac{1}{{\rm vol}_g(M)}\Big),
\end{equation}
where  $h$ is a smooth positive function, $N\in \mathbb N$ is a positive integer, $q_1,\cdots,q_N$ are distinct points on $M$ and where Dirac measures $\delta_{q_j}$ are located, $\gamma_j>-1$ are constants.\\
Solutions to \eqref{Q-singular} have the following geometric interpretation: Setting  $\tilde{g}:= e^{2u} g$ we obtain a metric conformal to $g$ on $M \setminus \{q_1, \cdots,q_N  \}$ which has a conical singularity at each $q_i$. One says that $\tilde{g}$ is represented by the divisor $D:= \sum_{i=1}^{N} \gamma_i q_i$. See Fang-Ma \cite{Fang-Ma}. Furthermore due to Gauss-Bonnet-Chern formula for conic four manifolds, see \cite{Chang-Qing-Yang-Duke}, \cite{Chang-Qing-Yang-Invent},\cite{BN19},   we have  that
$$
\tilde{\kappa}_P \, := \int_{M}Q_{\tilde{g}}{\rm d}V_{\tilde{g}} \, =  \, \int_{M}Q_g{\rm d}V_g \, + \, 8 \pi^2 \sum_{i=1}^{N} \gamma_i
$$
is a conformal invariant.\\
Considerable progress has been made for the regular case of (\ref{Q-singular}), that is $N=0$ in (\ref{Q-singular}). Under the assumption that the Kernel of the Paneitz operator contains only constants, Chang-Yang \cite{Chang-Yang} proved  existence for $\kappa_P<8\pi^2$, Djadli-Malchiodi \cite{Djadli-Malchiodi-2008} settled the case that $\kappa_P\neq 8\pi^2n$ for any $n\in \mathbb N$. Li-Li-Liu \cite{Li-Li-Liu} gave a necessary condition for existence in the case $\kappa_P = 8 \pi^2$, Ahmedou-Ndiaye\cite{Ahmedou-Ndiaye} developed a Morse theory  the case of $\kappa_P=8\pi^2n$ and Ndiaye \cite{Ndiaye-2} combined  the celebrated topological argument of Bahri-Coron \cite{Bahri-Coron} with the \emph{critical point theory at infinity } in \cite{Ahmedou-Ndiaye} to derive  some existence results.  We point out that  an  essential estimate related to the proof in \cite{Djadli-Malchiodi-2008} is a priori estimate when $\kappa_P$ is away from $8\pi^2 \mathbb N$ proved by Malchiodi in \cite{Malchiodi}.
 Later in \cite{Druet-Robert}, Druet and Robert extended such an a priori estimate to the following more general equation in the same class:
\begin{equation}\label{Q-equ-gen-reg}
P_g u+2b=2he^{4u},
\end{equation}
where $b$ is a smooth function. If $b=Q_g$ is the $Q$-curvature of the conformal metric $e^{2u}g$. More specially, assuming $h_k\to h_0$, $h_k\geq c_0>0$ and $b_k \to b_0$, then any sequence of solutions $\{u_k\}$ of (\ref{Q-equ-gen-reg}) with $h=h_k$ and $b=b_k$ is uniformly bounded under the condition $\int_{M}b_0{\rm d}V_g\neq8\pi^2n$, see also Malchiodi \cite{Malchiodi}.

However, bubbling can occur when $\int_{M}b_0{\rm d}V_g=8\pi^2n$ for some positive integer $n$. The understanding of this bubbling phnomenon is vital for the existence problem. The study of the blow-up profile and other blow-up phenomena for the Paneitz operator and other elliptic equations has attracted much interest recently and the reference is too numerous to be mentioned, we just list a few closely related to our article in our humble opinion:  \cite{Struwe-Robert, Brendle,Chang-Qing-Yang-3,Djadli-Malchiodi-2005,Fefferman,Gursky-Viaclovsky, hyder, Li-Li-Liu,Malchiodi-2,Malchiodi-Struwe,Ndiaye,Qing-Raske,Wei-1996,Wei-Xu,zhang-weinstein}. Particularly, in \cite{zhang-weinstein}, the third  named  author and Weinstein have obtained sharp estimates on the difference near the blow-up points between a bubbling sequence of solutions to (\ref{Q-equ-gen-reg}) with $h=h_k$ and $b=b_k$ and standard bubbles, and obtained the vanishing rate under the assumption that $(M,g)$ may not be locally conformally flat.

\medskip

When taking the singularities into the account as in (\ref{Q-singular}), we consider the following more general singular equation:
\begin{equation}\label{Q-equ-gen-singular}
P_gu+2b=2he^{4u}-8\pi^2\sum_{j=1}^{N}\gamma_j\Big(\delta_{q_j}-\frac{1}{vol_g(M)}\Big),
\end{equation}
where $h$ is a positive smooth function on $M$ and $b\in C^1(M)$.
Before stating our first main result, we define a {\itshape critical set} $\Gamma$ as follows:
\begin{equation*}
\Gamma=\Big\{16\pi^2n+16\pi^2\sum_{j\in J}(1+\gamma_j):\ \,n\in\mathbb{N}\cup\{0\}\ \,{\rm and}\ \,J\subset\{1,\cdots,N\}\Big\}.
\end{equation*}

In order to obtain the a priori estimates and existence results, we mainly study the blow-up phenomena for (\ref{Q-equ-gen-singular}). Let us consider the following equations:
\begin{equation}\label{Q-equation-blowup}
P_gu_k+2b_k=2h_ke^{4u_k}-8\pi^2\sum_{j=1}^{N}\gamma_j\Big(\delta_{q_j}-\frac{1}{{\rm vol}_g(M)}\Big)\quad {\rm in }\ \, M,
\end{equation}
with normalized total integration:
\begin{equation}\label{volume-normal}
	\int_M e^{4u_k}{\rm d}V_g=1.
\end{equation}

Let $\{u_k\}$ be a sequence of solutions of (\ref{Q-equation-blowup}). We say $p$ is a blowup point of $u_k$ if there exists a sequence $p_k\to p$ such that
$$u_k(p_k)+8\pi^2\sum_{j=1}^N\gamma_jG(p_k,q_j)\to \infty. $$ $u_k$ is called a sequence of blowup solutions if it has a blowup point. Here $G(x,p)$ is the Green's function of $P_g$ defined in (\ref{Green-func-expression}).
 For blowup solutions we assume that coefficient functions are regular enough to have limits:
\begin{equation}\label{assumption-coe}
\parallel b_k-b_0\parallel_{C^1(M)}\to 0,\quad \parallel h_k-h_0\parallel_{C^1(M)}\to 0,\quad 0<c_0<h_0<1/c_0.
\end{equation}
 Without loss of generality, we assume the integration of $h_ke^{u_k}$ is $1$:

Our first main result asserts that a priori estimate holds for $u_k$, as long as $2\int_Mb_k$ does not tend to the following critical set:

\begin{equation*}
\Gamma=\Big\{16\pi^2n+16\pi^2\sum_{j\in J}(1+\gamma_j):\ \,n\in\mathbb{N}\cup\{0\}\ \,{\rm and}\ \,J\subset\{1,\cdots,N\}\Big\}.
\end{equation*}

\begin{thm}[A Priori Estimate]\label{thm-apriori-est}
	
	Suppose (P) holds, $b$ and $h$ satisfy (\ref{assumption-coe}). If $\{u_k\}$ is a sequence of solutions of (\ref{Q-equation-blowup}) under restriction (\ref{volume-normal}) and $\int_M2b_0{\rm d}V_g\in\mathbb{R}^+\setminus \Gamma$,
	 $$\Big|u_k(x)+8\pi^2\sum_{j=1}^N\gamma_j G(x,q_j) \Big|\le C,\quad \forall x\in M $$
for some $C>0$ independent of $k$.
	\end{thm}

In particular, the a priori estimate holds for the singular prescribing $Q$-curvature equations. Indeed Theorem \ref{thm-apriori-est} is an extension of previous results of Malchiodi \cite{Malchiodi}, Druet-Robert \cite{Druet-Robert} and Fardoun-Regbaoui \cite{fard} for the regular  prescribed $Q$-curvature equation.  We point out that the argument in the regular case uses in a crucial way the explicit form of the bubble, while our argument uses only the asymptotic behavior of the bubble  and  is based on a Pohozaev identity for equations under conformal normal coordinates (see \cite{zhang-weinstein}).

If $\int_{M}2b_0dV_g\in \Gamma$ and $u_k$ does blowup around a singular source $q\in M$, our next result says that if $\gamma_q$ is not a positive integer, the Spherical Harnack inequality holds for $u_k$ around $q$:

\begin{thm}\label{sphe-har-uk} Suppose $\{u_k\}$ be a sequence of solutions of (\ref{Q-equation-blowup}) that also satisfies (\ref{volume-normal}),(P) holds and $b$, $h$ satisfy (\ref{assumption-coe}). If $q_j$ is a blowup point of $u_k$ and $\gamma_{q_j}\not \in \mathbb N$, there exist $C, \delta>0$ independent of $k$ such that
$$\max_{B(q_j,\delta)}(u_k(x)+\log |x-q_j|)\le C, \quad |x-q_j|\le \delta . $$
\end{thm}

Establishing Spherical Harnack inequality for bubbling solutions is critical for applications like a priori estimate, degree counting program and existence results. The readers may look at recent breakthroughs of the third author and Wei \cite{wei-zhang-adv} for Liouville equations.

One indispensable part of the blowup analysis for the Q curvature equation is the classification of global solutions on $\mathbb R^4$. For this purpose we consider
the limiting equation used to describe the profile of bubbling solutions:
\begin{equation*}
    \Delta^2 \tilde u=6e^{4 \tilde u}-8\pi^2\gamma \delta_0, \quad \int_{\mathbb R^4} e^{4\tilde u}<\infty, 
\end{equation*} 
where $\gamma>-1$ is a constant and the equation is defined on $\mathbb R^4$. Using 
$$ \Delta^2(\frac{1}{8\pi^2}\log \frac{1}{|x|})=\delta_0, \quad 
\mbox{and} \quad u(x)=\tilde u(x)-\gamma \log |x|, $$
we see that the equation for $\tilde u$ is equivalent to
\begin{equation}\label{equ-liou-2}
\Delta^2 u(x)=6|x|^{4\gamma}e^{4u(x)}, \quad  {\rm in} \ \; \mathbb{R}^4, \qquad
|x|^{4\gamma}e^{4u(x)}\in L^1(\mathbb{R}^4).
\end{equation}
Clearly if $u$ is a solution of (\ref{equ-liou-2}), so is $u_{\lambda}$ defined by
\begin{equation}\label{rescale}
u_{\lambda}(x)=u(\lambda x)+(1+\gamma)\log \lambda
\end{equation}
for any given $\lambda>0$. Our next main result is

\begin{thm}\label{thm-classification}
	Suppose that $u$ is a solution of (\ref{equ-liou-2}) with $\gamma>-1$ and $|u(x)|=o(|x|^2)$ at infinity. Then
	\begin{itemize}
		\item [(i)]
		$\int_{\mathbb{R}^4}6|y|^{4\gamma} e^{4u(y)}{\rm d}y=16\pi^2(1+\gamma)$.
		\item [(ii)]
		$
		u(x)=\frac{3}{4\pi^2}\int_{\mathbb{R}^4}\log\big(\frac{|y|}{|x-y|}\big)|y|^{4\gamma}e^{4u(y)}{\rm d}y+C_0
		$
for some $C_0\in \mathbb R$,
		\item [(iii)]
		Furthermore, if $-1<\gamma<0$, $u$ is radially symmetric about the origin and is unique up to scaling in (\ref{rescale}).
	\end{itemize}
	\end{thm}
Note that $(i)$ is a quantization result and is true for all $\gamma>-1$. $(iii)$ completes the classification for all $\gamma\in (-1,0]$. $(ii)$ is also proved as a part of the main theorem in \cite{hyder}. As a corollary of the classification in $(iii)$ we now derive a more specific description of the asymptotic behavior of $u$: For simplicity we use $\mu=1+\gamma$.
\begin{cor}\label{precise-u}
	Let $u$ be a solution of (\ref{equ-liou-2}) with $\gamma\in (-1,0)$ and suppose $|u(x)|=o(|x|^2)$ at $\infty$. Then for $M=[\frac{1}{4\mu}]$, there exists $c_0\in \mathbb R$ such that
	\begin{equation}\label{u-asy-p}
		u(x)=-2\mu\log|x|+c_0+\sum_{l=1}^M\frac{c_l}{|x|^{4l\mu }}+O(\frac{1}{|x|}),\quad |x|>1
\end{equation}
where $c_l=\frac{3e^{4\sum_{s=0}^{l-1}c_s}}{32 l^2\mu^2(1-2l\mu)(1+2l\mu)}$, $l=1,..,M$;
\begin{equation}\label{u-high-asy-p}
		-\Delta u(x)=\frac{4\mu}{|x|^2}+\sum_{l=1}^M\frac{4 c_l l\mu(2-4l \mu)}{|x|^{2+4l\mu}}+O(\frac{1}{|x|^3}),, \quad |x|>1	
\end{equation}
\end{cor}
Note that $[\frac{1}{4\mu}]$ stands for the largest integer no greater than $\frac{1}{4\mu}$. If $\gamma>-\frac{3}{4}$, $M=0$, the third term in (\ref{u-high-asy-p}) does not exist in this case.
If the $o(|x|^2)$ assumption in Theorem \ref{thm-classification} is removed, it is established in \cite{hyder} by Hyder et al that
if $u$ is a solution of (\ref{equ-liou-2}) with $\gamma>-1$, after an orthogonal transformation, $u(x)$ can be represented by
	\begin{equation}\label{rep-u-2}
	u(x)=\frac{3}{4\pi^2}\int_{\mathbb{R}^4}\log\big(\frac{|y|}{|x-y|}\big)|y|^{4\gamma}e^{4u(y)}{\rm d}y-\sum_{j=1}^{4}a_j(x_j-x_j^0)^2+c_0
\end{equation}
for some $c_0\in \mathbb R$. In this case we can prove a symmetry result: If $\gamma\in (-1,0)$ and $a_ix_i^0=0$ for all $i$, $u$ is a radial function. See Theorem \ref{thm-classification-2} for more detail.

\begin{rem} Our quantization and classification results directly point to the following important open questions:
1. Theorem \ref{thm-classification} does not cover the case that $\gamma>0$, except for the quantization result. The symmetry and profile for global solutions of (\ref{equ-liou-2}) in this case are largely unknown.

2. If the sub-quadratic growth assumption ($u(x)=o(|x|^2)$) is removed, it is not clear to us how the total integration of global solutions will change. It is proved in Lin \cite{lin-classification} that for $\gamma=0$ solutions with sub-quadratic growth give the largest integration. It is also proved by Hyder et al \cite{hyder} that for $\gamma<0$ there are non-radial solutions whose integration is greater than that of radial ones. It can be easily derived from our Theorem \ref{thm-classification-2} that if $\gamma<0$ and $a_ix_i^0=0$ for all $i$, the one with the sub-quadratic growth gives the largest integration. However, no information is known for other situations.
\end{rem}

When we were about to submit this article we heard that Jevnikar-Sire-Yang \cite{yang-wen} are working on a similar project independently and their results are to be posted soon.

Here we briefly outline  the strategy of the proofs   in our paper. For the proof of the classification result for globally defined singular equation, we follow the argument of Lin \cite{lin-classification} but we need to take care of all the complications caused by the singular source. In particular the method of moving planes relies crucially on the integral form of the global solutions.  We are able to prove the complete classification for $\gamma\in (-1,0]$ and a quantization result for all $\gamma>-1$. For blowup solutions we first use a small-energy lemma (Lemma \ref{lem-small-mass-regular}) to prove that there are at most finite blowup points. Then we take advantage of a Pohozaev identity established by Weinstein-Zhang \cite{zhang-weinstein} to describe a precise asymptotic behavior of blowup solutions around a blowup point. Then the total integration as well as precise asymptotic behavior of solutions can be further determined. With this information the critical set $\Gamma$ can be identified and if the total integral of $2b$ does not tend to $\Gamma$ we obtain a priori estimate. The idea to prove Theorem \ref{sphe-har-uk} is as follows: If the spherical harnack inequality is violated,
there should be finite small bubbling circles around the singular
source all tending to the singular source. Around each tiny bubbling
disk there is a Pohozaev identity, and a ``big" ball that includes
all these tiny balls also has a Pohozaev identity. The comparison of
these Pohozaev identities implies that the strength of the singular
source has to be an integer.

The organization of this paper is as follows. In section \ref{entire} we analyze the globally defined solutions and proved the quantization  and the classification results stated in Theorem \ref{thm-classification} and Corollary \ref{precise-u}. Then in Section \ref{preliminaries}, we list some useful facts about the conformal normal coordinates and Pohozaev identity and  in section \ref{blowup-local}, we perform a blow-up analysis near singular points. Section \ref{CC-Apriori} is devoted to a  concentration-compactness theorem and a priori estimate for the singular prescribing $Q$-curvature equation on 4-manifolds. Theorem \ref{spherical-Harnack} is also established in this section. Finally we provide is the appendix an useful estimate of the difference between the geodesic distance and the Euclidean one for nearby points on the manifold.

\section[Entire solutions]{Entire solutions of fourth order singular  Liouville type equations in $\mathbb{R}^4$}\label{entire}

In this section, we will follow the argument of Lin \cite{lin-classification} to analyze solutions of (\ref{equ-liou-2}) and prove Theorem \ref{thm-classification} and Theorem \ref{thm-classification-2}.

\subsection{Asymptotic behavior of entire solutions}

\quad

Our argument is progressive in nature and we shall obtain a rough estimate of $u$ at infinity. For this purpose we set
\begin{equation}\label{v-def}
    v(x):=\frac{3}{4\pi^2}\int_{\mathbb{R}^4}\log\big(\frac{|x-y|}{|y|}\big)|y|^{4\gamma}e^{4u(y)}{\rm d}y,
\end{equation}
which is obviously a solution of
\begin{equation}\label{v-equ}
\Delta^2v(x)=-6|x|^{4\gamma}e^{4u(x)},\quad{\rm in} \ \; \mathbb{R}^4.
\end{equation}
The asymptotic behavior of $u$ has a large to do with that of $v$, so in the first lemma we derive a rough upper bound of $v$.
 For convenience we set
\begin{equation}\label{energy-R4}
\alpha=\frac{3}{4\pi^2}\int_{\mathbb R^4}|y|^{4\gamma}e^{4u}dy.
\end{equation}

\begin{lem}\label{lem-v-upper}
	Suppose $u$ is a solution of (\ref{equ-liou-2}) and let $\alpha$ be given as in (\ref{energy-R4}). Then
	\begin{equation}\label{v-upper}
		v(x)\leq\alpha\log|x|+C
	\end{equation}
	for some constant $C$.
\end{lem}

\begin{proof}[\textbf{Proof}]
	Since the goal is to describe asymptotic behavior it is natural to assume $|x|>4$. For such $x$, we decompose $\mathbb{R}^4=A_1\cup A_2$, where
	\begin{equation*}
		A_1=\Big\{y:|y-x|\leq\frac{|x|}{2}\Big\},\quad A_2=\Big\{y:|y-x|\geq\frac{|x|}{2}\Big\}.
	\end{equation*}
	For $y\in A_1$, $\log\frac{|x-y|}{|y|}\leq 0$ because $|y|\geq|x|-|x-y|\geq\frac{|x|}{2}\geq|x-y|$, Thus
	\begin{equation*}
		\int_{A_1}\log\big(\frac{|x-y|}{|y|}\big)|y|^{4\gamma} e^{4u(y)}{\rm d}y\leq 0
	\end{equation*}
	and
	\begin{equation*}
	v(x)\leq\frac{3}{4\pi^2}\int_{A_2}\log\big(\frac{|x-y|}{|y|}\big)|y|^{4\gamma} e^{4u(y)}{\rm d}y\leq 0.
	\end{equation*}
	To evaluate the integral over $A_2$, we first make two trivial observations:
 $$|x-y|\leq|x|+|y|\leq|x||y|, \quad \mbox{if} \quad |y|\ge 2, $$
 $$\log|x-y|\leq\log|x|+C, \quad \mbox{if }\quad |y|\leq 2 $$
  where $|x|>4$ is used. Consequently
	\begin{equation*}
	\begin{split}
	v(x)&\leq\frac{3}{4\pi^2}\int_{A_2}\log\big(\frac{|x-y|}{|y|}\big)|y|^{4\gamma} e^{4u(y)}{\rm d}y\\
	&\leq\frac{3}{4\pi^2}\Big\{\log|x|\int_{A_2\cap\{|y|\geq2\}}|y|^{4\gamma} e^{4u}{\rm d}y+\int_{A_2\cap\{|y|\leq2\}}\log\big(\frac{|x-y|}{|y|}\big)|y|^{4\gamma} e^{4u}{\rm d}y\Big\}\\
	&\leq\frac{3}{4\pi^2}\Big\{\log|x|\int_{A_2}|y|^{4\gamma} e^{4u}{\rm d}y+C\int_{A_2\cap\{|y|\leq2\}}|y|^{4\gamma} e^{4u}{\rm d}y\\
	&\qquad\quad-\int_{A_2\cap\{|y|\leq2\}}\big(\log|y|\big)|y|^{4\gamma} e^{4u}{\rm d}y\Big\}\\
	&\leq\frac{3}{4\pi^2}\log|x|\int_{\mathbb{R}^4}|y|^{4\gamma} e^{4u}{\rm d}y+C.
	\end{split}
	\end{equation*}
	Lemma \ref{lem-v-upper} is established.
	\end{proof}

Before proving a lower bound of $v(x)$ we derive an expression of  $\Delta u(x)$ in integral form.

\begin{lem}\label{lem-lap-u}
	
	Suppose $u$ is a solution of (\ref{equ-liou-2}). Then there exists a constant $C_1\geq 0$ such that
	\begin{equation}\label{lap-u}
	\Delta u(x)=-\frac{3}{2\pi^2}\int_{\mathbb{R}^4}\frac{1}{|x-y|^2}|y|^{4\gamma}e^{4u(y)}{\rm d}y-C_1.
	\end{equation}
	
\end{lem}

\begin{proof}[\textbf{Proof}]
	
	Let $w(x)=u(x)+v(x)$. Then from the equations of $u$ and $v$ in (\ref{equ-liou-2}) and (\ref{v-equ}), we have $\Delta^2w=0$ in $\mathbb{R}^4$. Hence, $\Delta w$ is a harmonic function in $\mathbb{R}^4$. By the mean value property of harmonic functions, we have, for any $x_0\in \mathbb{R}^4$ and $r>0$,
	\begin{equation*}
	\begin{split}
	\Delta w(x_0)=\frac{2}{\pi^2 r^4}\int_{B(x_0,r)}\Delta w(y){\rm d}y=\frac{2}{\pi^2 r^4}\int_{\partial B(x_0,r)}\Delta w(y){\rm d}\sigma,
	\end{split}
	\end{equation*}
	where $\frac{\pi^2}{2}$ is the volume of the unit ball. That is
	\begin{equation}\label{lap-w}
	\frac{r}{4}\Delta w(x_0)=\dashint_{|y-x_0|=r}\frac{\partial w}{\partial r}(y){\rm d}\sigma,
	\end{equation}
	where $\dashint_{|y-x_0|=r}f(y){\rm d}\sigma=\frac{1}{2\pi^2r^3}\int_{|y-x_0|=r}f(y){\rm d}\sigma$ denotes the integral average of $f$ over $\partial B(x_0,r)$. Then integrating the identity above along $r$, we get
	\begin{equation*}
	\frac{r^2}{8}\Delta w(x_0)=\dashint_{|y-x_0|=r}w{\rm d}\sigma-w(x_0).
	\end{equation*}
	Therefore, the Jensen inequality implies
	\begin{equation*}
	\begin{split}
	\exp\Big(\frac{r^2}{2}\Delta w(x_0)\Big)&\leq e^{-4w(x_0)}\exp\Big(4\dashint_{|y-x_0|=r} w {\rm d}\sigma\Big)\\
	&\leq e^{-4w(x_0)}\dashint_{|y-x_0|=r} e^{4w} {\rm d}\sigma
	\end{split}
	\end{equation*}
	From Lemma \ref{lem-v-upper}, we have $w(x)=u(x)+v(x)\leq u(x)+\alpha\log|x|+C$, and as a consequence
	\begin{equation*}
	\begin{split}
	&\int_{0}^{\infty}r^{3-4\alpha+4\gamma}\exp\big(\frac{\Delta w(x_0)}{2}r^2\big){\rm d}r
	\leq\int_{\mathbb{R}^4}|x|^{-4\alpha+4\gamma}e^{-4w(x_0)}e^{4w(x)}{\rm d}x\\
	\leq& C\int_{\mathbb{R}^4}|x|^{-4\alpha+4\gamma}e^{4u(x)}|x|^{4\alpha}{\rm d}x
	=C\int_{\mathbb{R}^4}|x|^{4\alpha}e^{4u}{\rm d}x<+\infty,
	\end{split}
	\end{equation*}
	which means
	$$r^{3-4\alpha+4\gamma}\exp\Big(\frac{\Delta w(x_0)}{2}r^2\Big)\in L^1\big([1,+\infty)\big).$$
	From here we see  $\Delta w(x_0)\leq 0$ for all $x_0\in\mathbb{R}^4$. Liouville's Theorem implies that there exists some constant $C_1\geq0$ such that $\Delta w(x)\equiv-C_1$ in $\mathbb{R}^4$.
Lemma \ref{lem-lap-u} is established based on this and 	
	\begin{equation*}
	\Delta v(x)=\frac{3}{2\pi^2}\int_{\mathbb{R}^4}\frac{1}{|x-y|^2}|y|^{4\gamma}e^{4u(x)}{\rm d}y.
	\end{equation*}
\end{proof}

With the help of the representation for $\Delta u$, we can estimate $v(x)$ from below for $|x|$ large. We will use following result in Lemma 2.3 of \cite{lin-classification}.

\medskip

Let $h(x)$ be the solution of
\begin{equation*}
\left\{\begin{array}{lcl}
\Delta^2 h(x)=f(x), &&  {\rm in} \ \; \Omega,
\\
\Delta h(x)=h(x)=0, && {\rm on} \ \; \partial\Omega,
\end{array}
\right.
\end{equation*}
where $\Omega$ is a bounded domain of $\mathbb{R}^4$.

\begin{lemA}\label{lem-BM}\cite{lin-classification}
	
	Suppose $f\in L^1(\bar{\Omega})$.Then for any $\delta\in(0,32\pi^2)$, there exists a constant $C_{\delta}>0$ such that
	\begin{equation*}
	\int_{\Omega}\exp\Big(\frac{\delta|h|}{\parallel f\parallel_{L^1}}\Big){\rm d}x\leq C_{\delta}(diam \; \Omega)^4,
	\end{equation*}
	where $diam \;\Omega$ denotes the diameter of $\Omega$.
	
\end{lemA}

\begin{lem}\label{lem-v-lower}
	Let $u$ be a solution of (\ref{equ-liou-2}) and $v$ be in (\ref{v-def}). Then for given $\varepsilon>0$, there exists a constant $R=R(\varepsilon)$ only depending on $\varepsilon$ such that
	\begin{equation}\label{v-lower}
		v(x)\geq(\alpha-\varepsilon)\log|x|, \quad |x|>R(\epsilon).
	\end{equation}
	
\end{lem}

\begin{proof}[\textbf{Proof}]
	
	We first prove a claim slightly weaker than (\ref{v-lower}): for any $\varepsilon>0$, there exists $R=R(\varepsilon)>0$ such that
	\begin{equation}\label{v-lower-rough}
		v(x)\geq(\alpha-\frac{\varepsilon}{2})\log|x| +\frac{3}{4\pi^2}\int_{B(x,1)}(\log|x-y|)|y|^{4\gamma}e^{4u(y)}{\rm d}y.
	\end{equation}
	To prove (\ref{v-lower-rough}) we consider $\mathbb R^4$ as a disjoint union of three sets: $\mathbb{R}^4=A_1\cup A_2\cup A_3$, where
	\begin{align*}
		A_1&=\{y:|y|<R_0\}, \\
		A_2&=\{y:|x-y|\leq |x|/2,|y|\geq R_0\}, \\ A_3&=\{y:|x-y|\geq |x|/2,|y|\geq R_0\}.
	\end{align*}

	Then we choose $R_0=R_0(\varepsilon)$ sufficiently large so that
	\begin{equation*}
		\begin{split}
			&\frac{3}{4\pi^2}\int_{A_1}\log\big(\frac{|x-y|}{|y|}\big)|y|^{4\gamma}e^{4u(y)}{\rm d}y-\alpha \log|x|\\
			=&\frac{3}{4\pi^2}\log|x|\int_{A_1}\frac{\log|x-y|-\log|x|-\log|y|}{\log|x|}|y|^{4\gamma}e^{4u(y)}{\rm d}y-\frac{\varepsilon}{8} \log|x|  \\
			\geq& -\frac{\varepsilon}{4} \log|x|
		\end{split}
	\end{equation*}
	for large $|x|$. Thus we have
	\begin{equation}\label{int-A1}
		\frac{3}{4\pi^2}\int_{A_1}\log\big(\frac{|x-y|}{|y|}\big)|y|^{4\gamma}e^{4u(y)}{\rm d}y\geq (\alpha-\frac{\varepsilon}{4} )\log|x|.
	\end{equation}
	
	For $x\in A_2$ and $|x|$ large, we have $\frac{|x|}{2}\leq |x|\leq\frac{3}{2}|x|$. Then
	\begin{equation}\label{int-A2}
		\begin{split}
			&\int_{A_2}\log\big(\frac{|x-y|}{|y|}\big)|y|^{4\gamma}e^{4u(y)}{\rm d}y \\
			=&\int_{A_2}(\log|x-y|)|y|^{4\gamma}e^{4u(y)}{\rm d}y-\int_{A_2}(\log|y|)|y|^{4\gamma}e^{4u(y)}{\rm d}y \\
			\geq& \int_{B(x,1)}(\log|x-y|)|y|^{4\gamma}e^{4u(y)}{\rm d}y-\log(2|x|)\int_{A_2}|y|^{4\gamma}e^{4u(y)}{\rm d}y.
		\end{split}
	\end{equation}
	
	For $y\in A_3$, we use two trivial inequalities: $|x-y|\geq\frac{|x|}{2}\geq\frac{|y|}{4}$ if $|y|\le 2|x|$ and
  $|x-y|\geq|y|-|x|\geq\frac{|y|}{2}$ if $|y|\ge 2|x|$. Clearly in both cases, we have
	\begin{equation*}
		\frac{|x-y|}{|y|}\geq\frac{1}{4},\quad y\in A_3.
	\end{equation*}
	Therefore,
	\begin{equation}\label{int-A3}
		\frac{3}{4\pi^2}\int_{A_3}\log\big(\frac{|x-y|}{|y|}\big)|y|^{4\gamma}e^{4u(y)}{\rm d}y\geq \log\frac{1}{4}\int_{A_3}|y|^{4\gamma}e^{4u(y)}{\rm d}y.
	\end{equation}
	From (\ref{int-A1}), (\ref{int-A2}), (\ref{int-A3}) and $|y|^{4\gamma}e^{4u(y)}\in L^1(\mathbb{R}^4)$, we obtain (\ref{v-lower-rough}).
	
	\medskip
	
	Next, we show that
	\begin{equation}\label{int-B_x1-v}
		\int_{B(x,1)}(\log|x-y|)|y|^{4\gamma}e^{4u(y)}{\rm d}y\geq-C
	\end{equation}
	for some positive constant $C$. For this purpose we set
	\begin{equation}
		\tilde{u}(x)=u(x)+\gamma\log |x|
	\end{equation}
	Then $\tilde{u}$ satisfies
	\begin{equation}\label{equ-liou}
	\left\{\begin{array}{lcl}
	\Delta^2 \tilde{u}(x)=6e^{4\tilde{u}(x)}-8\pi^2\gamma\delta_0, &&  {\rm in} \ \; \mathbb{R}^4,
	\\
	e^{4\tilde{u}}\in L^1(\mathbb{R}^4). && \quad
	\end{array}
	\right.
	\end{equation}
	
	Let $0<\varepsilon_0<\pi^2$ and $R_0=R_0(\varepsilon_0)$ be sufficiently large such that
	\begin{equation}\label{energy-B4}
	6\int_{B(x,4)}|y|^{4\gamma}e^{4u(y)}{\rm d}y=6\int_{B(x,4)}e^{4\tilde{u}(y)}{\rm d}y\leq\varepsilon_0, \quad {\rm for}\ \, |x|\geq R_0,
	\end{equation}
	then we let $h$ be the solution of
	\begin{equation*}
	\left\{\begin{array}{lcl}
	\Delta^2 h(x)=6e^{4\tilde{u}(y)}, &&  {\rm in} \ \; B(x,4),
	\\
	h(x)=\Delta h(x)=0, && {\rm on} \ \; \partial B(x,4).
	\end{array}
	\right.
	\end{equation*}
	From Lemma \ref{lem-BM}, we can see that for $\varepsilon_0>0$ small,
	\begin{equation}\label{int-h-B4}
	\int_{B(x,4)}e^{24|h|}{\rm d}y\leq c_1
	\end{equation}
	for some constant $c_1$ independent of $x$.
	
	Next we set $q(y)=\tilde{u}(y)-h(y)$ for $y\in B(x,4)$, which clearly satisfies
	\begin{equation*}
	\left\{\begin{array}{lcl}
	\Delta^2 q(y)=0, &&  {\rm in} \ \; B(x,4),
	\\
	q(y)=\tilde{u}(y),\quad \Delta q(y)=\Delta\tilde{u}(y), && {\rm on} \ \; \partial B(x,4).
	\end{array}
	\right.
	\end{equation*}
	Let $\tilde{q}(y)=-\Delta q(y)$. We see that for $|x|$ large enough and $y\in\partial B(x,4)$
	\begin{equation*}
	\begin{split}
	\tilde{q}(y)=&-\Delta \tilde{u}(y)=-\Delta u(y)-\gamma\Delta(\log |y|)=-\Delta u(y)-\frac{2\gamma}{|y|^2}.
	\end{split}
	\end{equation*}
	By setting
	\begin{equation*}
		\hat{q}(y)=\tilde{q}(y)+\frac{2\gamma}{|y|^2}=-\Delta u(y)=\frac{3}{2\pi^2}\int_{\mathbb{R}^4}\frac{1}{|z-y|^2}|z|^{4\gamma}e^{4u(z)}{\rm d}z+C_1,\quad y\in\overline{B(x,2)},
	\end{equation*}
	we obviously have $\hat{q}(y)>0$ on $\partial B(x,4)$, and hence $\tilde{q}(y)>-2\gamma/|y|^2$ on $\partial B(x,4)$. Observing that $1/|y|^2$ is the fundamental solution of $\Delta$. In other words, $\hat{q}(y)$ is harmonic in $B(x,4)$ with positive boundary value on $\partial B(x,4)$. The maximum principle implies $\hat{q}>0$ in $B(x,4)$. Thus, by the Harnack inequality and mean value property of harmonic  functions, we have
	\begin{equation}\label{est-lap-q}
		\begin{split}
			\tilde{q}(y)&=\hat{q}(y)-\frac{2\gamma}{|y|^2}  
			\leq c_2\hat{q}(x)-\frac{2\gamma}{|y|^2}  
			=-c_2\dashint_{\partial B(x,4)}\Delta u{\rm d }\sigma-\frac{2\gamma}{|y|^2}  \\
			&=-c_2\dashint_{\partial B(x,4)}\Delta \tilde{u}{\rm d }\sigma+c_2\dashint_{\partial B(x,4)}\frac{2\gamma}{|y|^2} {\rm d }\sigma -\frac{2\gamma}{|y|^2}  \\
			&\le -c_2\dashint_{\partial B(x,4)}\Delta \tilde{u}{\rm d }\sigma+C,\quad y\in\overline{B(x,2)},
		\end{split}
	\end{equation}
	with constants $c_2$ and $C$.
	
	\medskip
	Integrating (\ref{equ-liou}) along $r$, we have
	\begin{equation*}
	\begin{split}
	&\dashint_{\partial B(x,4)}\Delta \tilde{u}{\rm d}\sigma-\Delta\tilde{u}(x)=\int_{0}^{r}\frac{3}{\pi^2 s^3}\int_{0}^{s}\int_{\partial B(x,t)}e^{4\tilde{u}}{\rm d}\sigma{\rm d}t{\rm d}s \\
	=&\int_{0}^{r}\frac{3}{\pi^2 s^3}\int_{0}^{s}t^3\int_{\partial B(x,1)}e^{4\tilde{u}}{\rm d}\sigma{\rm d}t{\rm d}s \\
	=&\int_{0}^{r}t^3\int_{\partial B(x,1)}e^{4\tilde{u}}\Big(\frac{3}{2\pi^2t^2}-\frac{3}{2\pi^2r^2}\Big){\rm d}\sigma{\rm d}t.
	\end{split}
	\end{equation*}
	That is
	\begin{equation}\label{ave-int-tildeu}
	\dashint_{\partial B(x,4)}\Delta \tilde{u}{\rm d}\sigma-\Delta\tilde{u}(x)=\frac{3}{2\pi^2}\int_{ B(x,r)}\Big(\frac{1}{|x-y|^2}-\frac{1}{r^2}\Big)e^{4\tilde{u}}{\rm d}y.
	\end{equation}
	Next by Lemma \ref{lem-lap-u} and (\ref{ave-int-tildeu}), we can see
	\begin{equation*}
	\begin{split}
	&-\dashint_{\partial B(x,4)}\Delta \tilde{u}{\rm d}\sigma=-\Delta\tilde{u}(x)-\frac{3}{2\pi^2}\int_{ B(x,r)}\frac{1}{|x-y|^2}e^{4\tilde{u}}{\rm d}y+ \frac{3}{2\pi^2r^2}\int_{ B(x,r)}e^{4\tilde{u}}{\rm d}y\\
	=&-\Delta u(x)-\frac{3}{2\pi^2}\int_{ B(x,r)}\frac{1}{|x-y|^2}e^{4\tilde{u}}{\rm d}y+ \frac{3}{2\pi^2r^2}\int_{ B(x,r)}e^{4\tilde{u}}{\rm d}y-\frac{2\gamma}{|x|^2}\\
	=&\frac{3}{2\pi^2r^2}\int_{|x-y|\geq r}\frac{1}{|x-y|^2}e^{4\tilde{u}}{\rm d}y+ \frac{3}{2\pi^2r^2}\int_{ B(x,r)}e^{4\tilde{u}}{\rm d}y-\frac{2\gamma}{|x|^2}+C_1.
	\end{split}
	\end{equation*}
	In particular, for $r=4$ and $|x|$ large,
	\begin{equation}\label{est-lap-aveu}
	-\dashint_{\partial B(x,4)}\Delta \tilde{u}{\rm d}\sigma\leq c_3.
	\end{equation}
	Hence, from (\ref{est-lap-q}), we get
	\begin{equation}\label{est-lap-q-2}
	\tilde{q}(y)\leq c_4,\,\quad y\in\overline{B(x,2)},
	\end{equation}
	and immediately
	$$|\tilde{q}(y)|\leq c_5,\,\quad y\in\overline{B(x,2)}.$$
	
	Since $q$ satisfies
	\begin{equation*}
	\left\{\begin{array}{lcl}
	\Delta q(y)=-\tilde{q}(y), &&  {\rm in} \ \; B(x,4),
	\\
	q(y)=\tilde{u}(y), && {\rm on} \ \; \partial B(x,4),
	\end{array}
	\right.
	\end{equation*}
	by estimates for linear elliptic equations, we have for any $p>1$ and $\sigma>2$,
	\begin{equation}\label{est-sup-q}
	\sup_{B(x,1)} q\leq c\big(\parallel q^{+}\parallel_{L^p(B(x,2))}+\parallel \tilde{q}\parallel_{L^{\sigma}(B(x,2))}\big),
	\end{equation}
	where $c=c(p,\sigma)$.
	
	On the other hand, we observe that $q^{+}(y)\leq\tilde{u}^{+}(y)+|h(y)|$ for $y\in B(x,4)$. Then by (\ref{int-h-B4}), we get
	\begin{equation*}
	\begin{split}
	\int_{ B(x,2)}(q^{+})^p\leq c_6\int_{ B(x,2)}e^{2q^{+}}\leq c_5\Big(\int_{ B(x,2)}e^{4\tilde{u}^{+}}\Big)^{\frac{1}{2}}\Big(\int_{ B(x,2)}e^{4|h|}\Big)^{\frac{1}{2}}\leq c_7\Big(\int_{ B(x,2)}e^{4\tilde{u}^{+}}\Big)^{\frac{1}{2}}
	\end{split}
	\end{equation*}
	Since $e^{4\tilde{u}^{+}}\leq 1+e^{4\tilde{u}}$, we have $\parallel q^{+}\parallel_{L^p(B(x,2))}\leq c_7$, which together with (\ref{est-lap-q-2}) and (\ref{est-sup-q}) implies
	\begin{equation}\label{est-sup-q-2}
	\sup_{B(x,1)} q\leq c_8.
	\end{equation}
	
	In view of $\tilde{u}=h+q$, we now obtain
	\begin{equation}
	\tilde{u}(y)\leq h(y)+q(y)\leq c_8+|h(y)|,\,\quad y\in\overline{B(x,2)}.
	\end{equation}
	Therefore,
	\begin{equation}\label{int-24u}
	\int_{ B(x,1)}e^{24\tilde{u}}\leq c_9	\int_{ B(x,1)}e^{24|h|}{\rm d}y\leq c_{10},
	\end{equation}
	Then
	\begin{equation*}
	\Big|\int_{ B(x,1)}(\log|x-y|)e^{4\tilde{u}}{\rm d}y\Big|\leq\Big(\int_{ B(x,1)}(\log|x-y|)^2{\rm d}y\Big)^{\frac{1}{2}}\Big(\int_{ B(x,1)}e^{8\tilde{u}(y)}{\rm d}y\Big)^{\frac{1}{2}}\leq c_{11},
	\end{equation*}
	which means
	\begin{equation}\label{int-minor-term}
	\Big|\int_{B(x,1)}(\log|x-y|)|y|^{4\gamma}e^{4u(y)}{\rm d}y\Big|\leq c_{11},
	\end{equation}
	where $c_{11}$ is a constant independent of $x$ ($|x|$ large). As a consequence, (\ref{v-lower-rough}) and (\ref{int-minor-term}) lead to
	\begin{equation*}
	v(x)\geq(\alpha-\frac{\varepsilon}{2})\log|x|-c_{11}\geq(\alpha-\varepsilon)\log|x|
	\end{equation*}
	for $|x|$ large, which is (\ref{v-lower}).
	

\end{proof}

With the estimates of $v(x)$ near infinity and the expression of $\Delta u$, we can show the expression of $u$ in integral form under the condition $|u(x)|=o(|x|^2)$ at $\infty$:

\begin{lem}\label{lem-rep-u}
	Suppose $|u(x)|=o(|x|^2)$ at $\infty$. Then there exists a constant $C_0$ such that
	\begin{equation}\label{rep-u}
		u(x)=\frac{3}{4\pi^2}\int_{\mathbb{R}^4}\log\big(\frac{|y|}{|x-y|}\big)|y|^{4\gamma}e^{4u(y)}{\rm d}y+C_0.
	\end{equation}
	Furthermore, for any given $\varepsilon>0$,
	\begin{equation}\label{u-est}
		-\alpha\log|x|-C\leq u(x)\leq(-\alpha+\varepsilon)\log|x|,\quad |x|\geq R(\varepsilon),
	\end{equation}
	where $R(\varepsilon)$ comes from Lemma \ref{lem-v-lower}.
\end{lem}

\begin{proof}[\textbf{Proof}]
	
	We start from the integral expression of $\Delta u$ in Lemma \ref{lem-lap-u}:
	\begin{equation*}
		\Delta u(x)=-\frac{3}{2\pi^2}\int_{\mathbb{R}^4}\frac{1}{|x-y|^2}|y|^{4\gamma}e^{4u(y)}{\rm d}y-C_1,\quad C_1\ge 0,
	\end{equation*}
	and we first prove $C_1=0$ by contradiction. If $C_1>0$  we have
	\begin{equation*}
		\Delta u(x)\leq -C_1<0,\quad |x|\geq R_0,
	\end{equation*}
	where $R_0$ is large.
	Let
	\begin{equation}
		h(y)=u(y)+\varepsilon|y|^2+A\big(|y|^{-2}-R_0^{-2}\big).
	\end{equation}
	Under the assumption of $|u(y)|=o(|y|^2)$ at $\infty$, we have $\lim\limits_{|y|\to +\infty}h(y)=+\infty$ for any fixed $\varepsilon>0$ and $A>0$. So we choose $\varepsilon>0$ small to make
	\begin{equation}
	\Delta h(y)=\Delta u(y)+8\varepsilon<-\frac{C_1}{2}<0,\quad |y|\geq R_0,
	\end{equation}
	and $A$ sufficiently large such that $\inf\limits_{|y|\geq R_0} h(y)$ is achieved by some $y_0\in \mathbb{R}^4$ and $|y_0|>R_0$. Clearly we have obtained a contradiction to the maximum principle. Hence, $C_1=0$ and $u+v$ is harmonic in $\mathbb{R}^4$.
	
	\medskip
	From Lemma \ref{lem-v-upper} and Lemma \ref{lem-v-lower}, we know for $|x|$ large enough
	\begin{equation*}
		(\alpha-\varepsilon)\log|x|\leq v(x)\leq \alpha\log|x|+C,
	\end{equation*}
	which together with the assumption $|u(x)|=o(|x|^2)$ at $\infty$ indicates
	\begin{equation}
		|u(x)+v(x)|=o(|x|^2)\quad {\rm at} \ \,\infty.
	\end{equation}
	Since $u+v$ is harmonic, we have
	\begin{equation*}
		u(x)+v(x)=\sum_{j=1}^{4}a_j x_j+a_0
	\end{equation*}
	with some constants $a_j\in\mathbb{R},\,j=0,\cdots,4$. Therefore, for $|x|$ large enough, we get
	\begin{equation*}
		e^{4u(x)}=e^{a_0}e^{-4v(x)}e^{\sum_{j=1}^4a_j x_j}\geq C|x|^{-4\alpha}e^{\sum_{j=1}^4 a_j x_j}.
	\end{equation*}
	Since $|y|^{4\gamma}e^{4u(x)}\in L^1(\mathbb{R}^4)$, we have $a_j=0$ for $1\leq j\leq 4$. Therefore,
	\begin{equation*}
		u(x)=-v(x)+a_0=\frac{3}{4\pi^2}\int_{\mathbb{R}^4}\log\big(\frac{|y|}{|x-y|}\big)|y|^{4\gamma}e^{4u(y)}{\rm d}y+C_0,
	\end{equation*}
	and then
	\begin{equation*}
	-\alpha\log|x|-C\leq u(x)\leq(-\alpha+\varepsilon)\log|x|,
	\end{equation*}
	for $|x|$ large.	Lemma \ref{lem-rep-u} is established.
	\end{proof}

Next we need a Pohozaev identity for $u$ of
\begin{equation}\label{equ-PI}
	\Delta^2 u=Q(x)e^{4u}\ \,{\rm in} \ \,\mathbb{R}^4.
\end{equation}

\begin{lem}\label{lem-PI}
	Suppose $u$ is an entire smooth solution of (\ref{equ-PI}). Then for any bounded domain, we have
	\begin{equation}\label{PI-omega}
	\begin{split}
	   &\int_{\Omega}\big(Qe^{4u}+\frac{1}{4}<x,\nabla Q>e^{4u}\big){\rm d}x \\
	   =&\frac{1}{4}\int_{\partial\Omega}<x,\nu>Q(x) e^{4u}{\rm d}\sigma+\int_{\partial\Omega}\Big\{\frac{1}{2}|\Delta u|^2<x,\nu>-2\frac{\partial u}{\partial\nu}\Delta u \\
	   &\quad -<x,\nabla u>\frac{\partial\Delta u}{\partial\nu}-<x,\nabla\Delta u>\frac{\partial u}{\partial \nu}+<x,\nu><\nabla u,\nabla \Delta u>\Big\}{\rm d}\sigma.
	\end{split}
	\end{equation}
	In particular, taking $\Omega=B_R$, we have
	\begin{equation}\label{PI-BR}
	\begin{split}
	&\int_{B_R}Q(x)e^{4u}{\rm d}x+\frac{1}{4}\int_{B_R}<x,\nabla Q>e^{4u}{\rm d}x \\
	=&\frac{1}{4}\int_{\partial B_R}|x|Q e^{4u}{\rm d}\sigma+\frac{1}{2}\int_{\partial B_R}|x||\Delta u|^2{\rm d}\sigma-2\int_{\partial B_R}\frac{\partial u}{\partial r}\Delta u{\rm d}\sigma -\int_{\partial B_R}|x|\frac{\partial u}{\partial r}\frac{\partial \Delta u}{\partial r}{\rm d}\sigma .
	\end{split}
	\end{equation}
\end{lem}

\begin{proof}[\textbf{Proof}]
	
	Multiplying (\ref{equ-PI}) by $x\cdot\nabla u$, we have
	\begin{equation}\label{IBP}
		\int_{\Omega}\Delta^2 u(x\cdot\nabla u)=\int_{\Omega}Q(x)e^{4u}(x\cdot\nabla u).
	\end{equation}
	After integrating by parts and direct computation, we get
	\begin{equation*}
		\begin{split}
		({\rm RHS}) \; {\rm of}\; (\ref{IBP})=\frac{1}{4}\int_{\partial\Omega}<x,\nu>Q(x) e^{4u}{\rm d}\sigma-\int_{\Omega}\big(Qe^{4u}+\frac{1}{4}<x,\nabla Q>e^{4u}\big){\rm d}x,
		\end{split}
	\end{equation*}
	and
	\begin{equation*}
	\begin{split}
	({\rm LHS}) \; {\rm of}\; (\ref{IBP})=&-\frac{1}{2}\int_{\partial\Omega}|\Delta u|^2<x,\nu>+2\int_{\partial\Omega}\frac{\partial u}{\partial\nu}\Delta u+\int_{\partial\Omega}<x,\nabla u>\frac{\partial\Delta u}{\partial\nu}\\
	&+\int_{\partial\Omega}<x,\nabla\Delta u>\frac{\partial u}{\partial \nu}-\int_{\partial\Omega}<x,\nu><\nabla u,\nabla \Delta u>.
	\end{split}
	\end{equation*}
	Thus we establish (\ref{PI-omega}). Taking $\Omega=B_R$, we immediately obtain (\ref{PI-BR}) from (\ref{PI-omega}). Note that the singularity at the origin is insignificant that contributes nothing to the final form of the Pohozaev identity.
	
\end{proof}

From the Pohozaev identity we shall determine the exact value of $\alpha$.

\begin{lem}\label{alpha}
	Let $u$ be a solution of (\ref{equ-liou-2}). Assume $|u(x)|=o(|x|^2)$ at $\infty$, then $\alpha=2(1+\gamma)$.
\end{lem}

\begin{proof}[\textbf{Proof}]
	
	Taking $Q(x)=6|x|^{4\gamma}$ in (\ref{PI-BR}), we have
	\begin{equation}\label{PI-BR-2}
		\begin{split}
		&6(1+\gamma)\int_{B_R}|x|^{4\gamma}e^{4u}{\rm d}x \\
		=&\frac{1}{4}\int_{\partial B_R}6r^{4\gamma+1}e^{4u}{\rm d}\sigma+\frac{1}{2}\int_{\partial B_R}r|\Delta u|^2-2\int_{\partial B_R}\frac{\partial u}{\partial r}\Delta u -\int_{\partial B_R}r\frac{\partial u}{\partial r}\frac{\partial \Delta u}{\partial r} .
		\end{split}
	\end{equation}
	
	In view of Lemma \ref{lem-rep-u}, we have obtained
	\begin{equation*}
		u(x)=\frac{3}{4\pi^2}\int_{\mathbb{R}^4}\log\big(\frac{|y|}{|x-y|}\big)|y|^{4\gamma}e^{4u(y)}{\rm d}y+C_0,
	\end{equation*}
	and $e^{u(y)}\geq |y|^{-4\alpha}$ for $|y|$ large enough. Then $|y|^{4\gamma}e^{4u(y)}\geq|y|^{4(\gamma-\alpha)}$. Hence $|y|^{4\gamma}e^{4u(y)}\in L^1(\mathbb{R}^4)$ implies $\alpha>1+\gamma$. On the other hand, by the representation of $u$ and direct calculations, there hold
	\begin{equation*}
		\frac{\partial u}{\partial r}(x)=-\frac{3}{4\pi^2}\int_{\mathbb{R}^4}\frac{x\cdot(x-y)}{|x||x-y|^2}|y|^{4\gamma}e^{4u(y)}{\rm d}y,
	\end{equation*}
	\begin{equation*}
	\Delta u(x)=-\frac{3}{2\pi^2}\int_{\mathbb{R}^4}\frac{1}{|x-y|^2}|y|^{4\gamma}e^{4u(y)}{\rm d}y,
	\end{equation*}
	and
	\begin{equation*}
	\frac{\partial}{\partial r}\Delta u(x)=\frac{3}{\pi^2}\int_{\mathbb{R}^4}\frac{x\cdot(x-y)}{|x||x-y|^4}|y|^{4\gamma}e^{4u(y)}{\rm d}y,
	\end{equation*}
	Recall the definication of $u$, Lemma \ref{lem-rep-u} and $\alpha>1+\gamma$, then we have
	\begin{equation}\label{u-limit}
	    \begin{split}
	    &\lim_{r\to +\infty}\frac{\partial u}{\partial r}=0, \quad
	    \lim_{r\to +\infty}r\frac{\partial u}{\partial r}=-\alpha,\\
	    &\lim_{r\to +\infty}r^2\Delta u=-2\alpha, \quad
	    \lim_{r\to +\infty}r^3\frac{\partial }{\partial r}\Delta u=4\alpha,
	\end{split}
	\end{equation}
	where $r=|x|$. Therefore, applying the Pohozaev identity (\ref{PI-BR-2}) and (\ref{u-limit}), we have
	\begin{equation*}
		8\pi^2(1+\gamma)\alpha=4\pi^2\alpha^2,
	\end{equation*}
	which leads to $\alpha=2(1+\gamma)$.
	
\end{proof}

Now we can determine the asymptotic behavior of $u$ at infinity using the exact value of $\alpha$.
\begin{lem}\label{lem-u-asy}
	Let $u$ be a solution of (\ref{equ-liou-2}) and suppose $|u(x)|=o(|x|^2)$ at $\infty$. Then there exist $c_0\in \mathbb R$, $\tau>0$ such that for $|x|>1$,
\begin{equation}\label{u-asy}
u(x)=-2(1+\gamma)\log |x|+c_0+O(|x|^{-\tau}),\quad |x|>1
 \end{equation}
and
\begin{equation}\label{u-high-asy}
\nabla^j(u(x)+2(1+\gamma)\log |x|)=O(|x|^{-\tau-j}), \quad j=1,2,3,4,\quad |x|>1.
\end{equation}
In particular
\begin{equation}\label{lap-u-asy}
-\Delta u(x)=\frac{4(1+\gamma)}{|x|^2}+O(|x|^{-2-\tau}),\quad |x|>1.
\end{equation}
\end{lem}

\begin{proof}[\textbf{Proof}]
	
	Let $w(x)=u(\frac{x}{|x|^2})-2(1+\gamma)\log|x|$, then from the equation of $u$ and the assumption, we see $w$ satisfies
	\begin{equation}\label{equ-w}
	\left\{\begin{array}{ll}
	\Delta^2 w(x)=6|x|^{4\gamma}e^{4w(x)},\quad {\rm in} \ \, \mathbb{R}^4\setminus\{0\},	\\
	|w(x)|=o(\log\frac{1}{|x|}),\quad |\Delta w(x)|=o(\frac{1}{|x|^2}),\quad{\rm as}\ \, |x|\to 0.
	\end{array}
	\right.
	\end{equation}
	Let $h(x)$ be a solution of
	\begin{equation}
	\left\{\begin{array}{lcl}
	\Delta^2 h(x)=6|x|^{4\gamma}e^{4w(x)},&& {\rm in} \ \, B_1\setminus \{0\},	\\
	h(x)=w(x),\quad \Delta h(x)=\Delta w(x),&&{\rm on} \ \, \partial B_1.
	\end{array}
	\right.
	\end{equation}
	and $q(x)=w(x)-h(x)$. Then $q(x)$ satisfies
	 \begin{equation}
	 \left\{\begin{array}{lcl}
	 \Delta^2 q(x)=0,&& {\rm in} \ \, B_1,	\\
	 q(x)=\Delta q=0,&& {\rm on} \ \, \partial B_1,\\
	 |q(x)|=o(\log\frac{1}{|x|}),\quad |\Delta q(x)|=o(\frac{1}{|x|^2}),&& {
	 \rm as}\ \, |x|\to 0.
	 \end{array}
	 \right.
	 \end{equation}
	 First for $\Delta q$, since its growth near the singular source is weaker than fundamental solutions, the singularity is removable, thus $\Delta q=0$ in $B_1$.
By exactly the same reason we further conclude that $q\equiv 0$ in $B_1$.  That means $w(x)=h(x)\in C^{0,\tau}(\bar{B}_1)$. It suffices to consider the regularity of $h$ in $B_1$.
	
	 Note that
	 \begin{equation*}
	 \begin{split}
	 |x|^{4\gamma}e^{4w(x)}&=|x|^{4\gamma}e^{4u(\frac{x}{|x|^2})-8(1+\gamma)\log|x|} \\
	 &\sim|x|^{4\gamma}|x|^{-8(1+\gamma)}|x|^{4\alpha} \sim|x|^{4\gamma},\quad {\rm near} \ \, 0,
	 \end{split}
	 \end{equation*}
	 where we used Lemma \ref{lem-rep-u} and Lemma \ref{alpha}.
	 	
By standard elliptic estimate, if $\gamma>-\frac{3}{4}$, $w\in C^{1,\tau}$ for some $\tau>0$, if $-1<\gamma\le -\frac{3}{4}$ we have $w\in C^{0,\tau_1}$ for $\tau_1<4(1+\gamma)$. So in either case we use $\tau\in (0,1)$ to have the following expansion of $u$:
$$u(x)=-2(1+\gamma)\log |x|+c_0+O(|x|^{-\tau}),\quad |x|>1. $$
Based on this we can use standard elliptic estimate to obtain corresponding gradient estimates: now $|x|^{4\gamma}e^{4u}$ can be written as
\begin{equation}\label{b-exp}
|x|^{4\gamma}e^{4u}=e^{c_0}r^{-8-4\gamma}+O(r^{-8-4\gamma-\tau}),\quad r=|x|>1
\end{equation}
Using this in the expression of $u(x)$ in (\ref{rep-u}) and $-\Delta u$ in (\ref{lap-u}) it is easy to obtain (\ref{lap-u-asy}) for $\Delta u$. Then the asymptotic behavior for other derivatives in (\ref{u-high-asy}) is a direct consequence of standard elliptic estimates.

\end{proof}

\begin{rem} \label{remark-mon}It is important to observe from (\ref{u-high-asy}) that for $r$ large, 
$$\partial_r(u(x)+\gamma \log |x|)=-\frac{2+\gamma}r+O(r^{-1-\tau}),\quad r=|x|, $$
thus the function $|x|^{4\gamma}e^{4u(x)}$ is strictly decreasing in $r=|x|$ for $r$ large.
\end{rem}

\subsection{Classification of entire solutions for the case $-1<\gamma<0$}

\quad

In this subsection, we will show the solution of (\ref{equ-liou}) or (\ref{equ-liou-2}) has radial symmetry and uniqueness property up to scaling if $-1<\gamma<0$. Similar to \cite{lin-classification}, we will use the method of moving planes. But the situation for singular equation is a lot harder since it is difficult to obtain precise asymptotic behavior of $u$ without knowing its radial symmetry. In this proof the integral expressions of $u$ and $\Delta u$ play a crucial role.

Suppose that $u$ is a smooth entire solution of (\ref{equ-liou-2}) with $|u(x)|=o(|x|^2)$ at $\infty$. Recall $-\Delta u>0$ in $\mathbb{R}^4$ and (\ref{u-asy})$\sim$(\ref{u-high-asy}), so we will apply the method of moving planes to $-\Delta u$. Let $v(x)=-\Delta u(x)$. Then by Lemma \ref{lem-u-asy},
\begin{equation}\label{harmonic-asy}
v(x)=\frac{4(1+\gamma)}{|x|^2}+O(|x|^{-2-\tau}),\quad |x|>1.
\end{equation}

First, we state some conventional notations for moving planes. For any $\lambda\in\mathbb{R}$, let $T_{\lambda}=\{x\in\mathbb{R}^4:x_1=\lambda\}$, $\Sigma_{\lambda}=\{x:x_1>\lambda\}$ and $x^{\lambda}=(2\lambda-x_1,x_2,x_3,x_4)$ be the reflection point of $x$ with respect to $T_{\lambda}$.

\begin{proof}[\textbf{Proof of Theorem \ref{thm-classification}}]

Lemma \ref{alpha} establishes (i) in Theorem \ref{thm-classification}. Next we aim to prove the radial symmetry of solutions by the method of moving planes.

\textbf{Step 1:}
We start moving planes along $x_1$-direction. For any $\lambda$, we consider $w_{\lambda}(x)=u(x)-u(x^{\lambda})$ in $\Sigma_{\lambda}$. Then $w_{\lambda}(x)$ satisfies
\begin{equation}\label{equ-MP}
\left\{\begin{array}{lcl}
\Delta^2 w_{\lambda}(x)=b_{\lambda}(x)w_{\lambda}(x), &&  {\rm in} \ \; \Sigma_{\lambda},
\\
w_{\lambda}(x)=\Delta w_{\lambda}(x)=0, && {\rm on} \ \; \partial T_{\lambda},
\end{array}
\right.
\end{equation}
where
\begin{equation*}
	b_{\lambda}(x)=6\frac{|x|^{4\gamma}e^{4u(x)}-|x^{\lambda}|^{4\gamma}e^{4u(x^{\lambda})}}{u(x)-u(x^{\lambda})}.
\end{equation*}

First we claim that there exists $\lambda_0<0$ such that $w_{\lambda_0}>0$ and $v(x)-v_{\lambda_0}(x)> 0$ in $\Sigma_{\lambda_0}$.
Using the expression of $u$ in (\ref{rep-u}) we can obviously write $u(x)-u_{\lambda}(x)$ as
\begin{equation}\label{est-w-2}
u(x)-u_{\lambda}(x)=\frac{3}{4\pi^2}\int_{\mathbb R^4}\log \frac{|x^{\lambda}-y|}{|x-y|}|y|^{4\gamma}e^{4u(y)}dy=\int_{\Sigma_{\lambda}}+\int_{\mathbb R^4\setminus
\Sigma_{\lambda}}, 
\end{equation}
where $u_{\lambda}(x)=u(x^{\lambda})$. Changing the integration over $\mathbb R^4\setminus \Sigma_{\lambda}$ to $\Sigma_{\lambda}$ by a change of variable:
$z=(z_1,z_2,z_3,z_4)=(2\lambda-y_1,y_2,y_3,y_4)$ we see that
\begin{equation}\label{nice-t}
u(x)-u_{\lambda}(x)=\frac{3}{4\pi^2}\int_{\Sigma_{\lambda}}\log \frac{|x^{\lambda}-y|}{|x-y|}(|y|^{4\gamma}e^{4u(y)}-|y^{\lambda}|^{4\gamma}e^{4u(y^{\lambda})})dy.
\end{equation}
Here we remark that writing $w_{\lambda}$ in the form of  (\ref{nice-t}) is crucial for our argument, it prevents us from very delicate asymptotic analysis for $u$ and its derivatives.
From (\ref{nice-t}) we claim that when $\lambda$ is very negative and when $x\in \Sigma_{\lambda}$, $u-u_{\lambda}>0$ in $\Sigma_{\lambda}$. Indeed, first it is obvious that $\log (|x^{\lambda}-y|/|x-y|)>0$. Next we observe the integration of $y$ over $B_R$ and outside $B_R$, respectively. For integration over $B_R$, we clearly have $|y|^{4\gamma}e^{4u}>c_0>0$ for some $c_0>0$. Based on the asymptotic behavior of $u$ we can easily make $|y^{\lambda}|^{4\gamma}e^{4u(y^{\lambda})}<c_0$ for $\lambda$ very negative.  For $y\in \Sigma_{\lambda}\setminus B_R$,
by Remark \ref{remark-mon} we have
$$|y|^{4\gamma}e^{4u(y)}>|y^{\lambda}|^{4\gamma}e^{4 u(y^{\lambda})}
\mbox{ because }\,\, |y^{\lambda}|>|y|. $$
We also observe that for $v=-\Delta u$,
$$v(x)-v_{\lambda}(x)=\frac{3}{2\pi^2}\int_{\mathbb R^4}(\frac{1}{|x-y|^2}-\frac{1}{|x^{\lambda}-y|^2})|y|^{4\gamma}e^{4u(y)}dy. $$
After a similar transformation we have
$$v(x)-v_{\lambda}(x)=\frac{3}{2\pi^2}\int_{\Sigma_{\lambda}} (\frac{1}{|x-y|^2}-\frac{1}{|x^{\lambda}-y|^2})(|y|^{4\gamma}e^{4u(y)}
-|y^{\lambda}|^{4\gamma}e^{4u(y^{\lambda})})dy. $$
By exactly the same reasoning we see that for $\lambda$ very negative, $v(x)>v_{\lambda}(x)$ in $\Sigma_{\lambda}$.

Let $\lambda_0$ be the starting position for the moving plane process and let $\bar \lambda\le 0$ be the upper limit:
\begin{equation}\label{limit-location}
	\bar \lambda:=\sup\big\{\lambda\le 0:v(x^{\mu})<v(x)\quad{\rm for\ \,all}\ \,x\in\Sigma_{\mu}\ \,{\rm and}\ \,\mu\leq\lambda\big\}.
\end{equation}
 Next we claim that $\bar \lambda=0$. If this is not the case, we have $\bar \lambda<0$. From the equation for $w_{\bar \lambda}$:
 \begin{equation}\label{w-0}
\Delta^2w_{\bar \lambda}(x)=6|x^{\bar \lambda}|^{4\gamma}(e^{4u(x)}-e^{4u_{\bar \lambda}(x)})+6(|x|^{4\gamma}-|x^{\bar \lambda}|^{4\gamma})e^{4u(x)},\quad x\in \Sigma_{\bar\lambda}.
\end{equation}
we see that on one hand the continuity gives $\Delta w_{\bar \lambda}(x)\leq 0$, since $w_{\bar \lambda}\to 0$ as $|x|\to+\infty$ and $w_{\bar \lambda}\big|_{T_{\bar \lambda}}=0$. On the other hand we have $|x^{\bar \lambda}|>|x|$. This strict inequality and the strong maximum principle combined gives
$$w_{\bar \lambda}(x)>0,\quad \Delta w_{\bar \lambda}(x)<0,\quad \mbox{ in }\quad \Sigma_{\bar \lambda}. $$
Then we claim that for $\epsilon>0$ small we still have 
\begin{equation}\label{vio-lambda}
v(x^{\mu})<v(x), \quad \mbox{for all } x\in \Sigma_{\mu} \quad \mbox{ and } \mu\le \bar \lambda+
\epsilon. 
\end{equation}
Clearly once (\ref{vio-lambda}) is verified, we obtain a contradiction to the definition of $\bar \lambda$. 
To prove (\ref{vio-lambda}) we first make a trivial observation: For any fixed $R>>1$, Hopf lemma and $w_{\bar \lambda}>0$ in $\Sigma_{\bar\lambda}\cap B_R$ means that if $\lambda$ is slightly greater than $\bar \lambda$, we still have $w_{\lambda}>0$ in $\Sigma_{\lambda}\cap B_R$. Thus we only consider $|x|>R$. We use two different expressions of $w_{\lambda}$: One is (\ref{est-w-2}), which will be used for crude estimate, the other one is based on (\ref{nice-t}):
\begin{equation}\label{w-more}
w_{\lambda}(x)=\frac 3{4\pi^2}\int_{\Sigma_{\lambda}}\log \frac{|x^{\lambda}-y|}{|x-y|}
\bigg (|y|^{4\gamma}(e^{4u}-e^{4u_{\lambda}})+(|y|^{4\gamma}-|y^{\lambda}|^{4\gamma})e^{4u}\bigg )dy.
\end{equation}
Similarly for $v-v_{\lambda}$ we have
\begin{align}\label{v-more}
&v(x)-v_{\lambda}(x)\\
=&\frac{3}{2\pi^2}\int_{\Sigma_{\lambda}} (\frac{1}{|x-y|^2}-\frac{1}{|x^{\lambda}-y|^2})
\bigg (|y|^{4\gamma}(e^{4u}-e^{4u_{\lambda}})+(|y|^{4\gamma}-|y^{\lambda}|^{4\gamma})e^{4u}\bigg )dy. \nonumber
\end{align}
Here it is important to point out that since $\bar \lambda<0$, a perturbation of $\epsilon$ still satisfies $\bar \lambda+\epsilon<0$.
Writing $w_{\lambda}=w_{\lambda}^+-w_{\lambda}^-$, our goal is to prove that $w_{\lambda}^-\equiv 0$ for $\lambda$ slightly greater than $\bar \lambda$. We claim that there exist $\epsilon,C>0$ such that
\begin{equation}\label{w-minus}
w_{\lambda}^-(x)\le C|x|^{-1-\epsilon}, \quad \mbox{for}\quad |x|>1
\end{equation}

The proof of (\ref{w-minus}) is by iteration. Note that $w_{\lambda}^-(x)=0$ when $|x|<R$. First using (\ref{est-w-2}) we have 
$$w_{\lambda}(x)=\frac{3}{8\pi^2}\int_{\mathbb R^4}\log (1+\frac{|x^{\lambda}-y|^2-|x-y|^2}{|x-y|^2})|y|^{4\gamma}e^{4u}dy=
\int_{E_1}+\int_{E_2}+\int_{E_3}. $$
where $E_1=B(0,|x|/2)$, $E_2=B(x,|x|/2)$ and $E_3=\mathbb R^4\setminus (E_1\cup E_2)$. The estimate on $E_1$ is
$$|\int_{E_1}|\le \frac{3}{8\pi^2}\int_{E_1}|\log (1+\frac{4(y_1-\lambda)(x_1-\lambda)}{|x-y|^2})|y|^{4\gamma}e^{4u}|dy. $$
Since $|x-y|\sim |x|$ it is easy to obtain an upper bound of $O(|x|^{-4\mu})$. The integration on $E_2$ and $E_3$ has an upper bound $O(|x|^{-4\mu+\epsilon})$. 
Thus
\begin{equation}\label{w-minus-2}
w_{\lambda}^-(x)\le C|x|^{-4\mu+\epsilon}. 
\end{equation}
Using (\ref{w-minus-2}) in (\ref{w-more}) we have
$$w_{\lambda}(x)\ge -\frac{3}{4\pi^2}\int_{\Sigma_{\lambda}}\log \frac{|x^{\lambda}-y|}{|x-y|}4|y|^{4\gamma}e^{4\xi}w_{\lambda}^-(y)dy. 
$$
where $\xi$ comes from the mean value theorem. For each $x$ satisfying $w_{\lambda}(x)\le 0$, we obtain from standard estimates that
$$w_{\lambda}^-(x)\le C|x|^{-8\mu+\epsilon}. $$
After finite steps we have (\ref{w-minus}). 

Now we use (\ref{w-minus}) to evaluate $v-v_{\lambda}$:
\begin{align*}
&v(x)-v_{\lambda}(x)\\
\ge &\frac{3}{2\pi^2}\int_{\Sigma_{\lambda}}(\frac{1}{|x-y|^2}-\frac{1}{|x^{\lambda}-y|^2})(4|y|^{4\gamma}e^{4\xi}w_{\lambda}^-(y)+(|y|^{4\gamma}-|y^{\lambda}|^{4\gamma})e^{4u}. \\
=&\frac{6}{2\pi^2}\int_{\Sigma_{\lambda}}\frac{(y_1-\lambda)(x_1-\lambda)}{|x-y|^2|x^{\lambda}-y|^2}(4|y|^{4\gamma}e^{4\xi}w_{\lambda}^-(y)+(|y|^{4\gamma}-|y^{\lambda}|^{4\gamma}))e^{4u}dy. 
\end{align*}
Note that $|y|^{4\gamma}-|y^{\lambda}|^{4\gamma}>0$ in $\Sigma_{\lambda}$ and $e^{4u}>c$ in $\Sigma_{\lambda}\cap B_R$ for some positive $c$. After integrating the second term, which is positive, we have
$$v(x)-v_{\lambda}(x)\ge -\frac{6}{\pi^2}\int_{\Sigma_{\lambda}} (\frac{(y_1-\lambda)(x_1-\lambda)}{|x-y|^2|x^{\lambda}-y|^2})
 4|y|^{4\gamma}e^{4\xi}w_{\lambda}^-(y))dy+c_0(x_1-\lambda)|x|^{-4} $$
for $|x|$ large. Then using (\ref{w-minus}) in the evaluation of the first term we see that $v-v_{\lambda}$ is positive even for $\lambda$ slightly larger than $\bar \lambda$. This is certainly a contradiction to the definition of $\bar \lambda$. Thus we have proved that $\bar \lambda=0$, which
means $u(x)\ge u(x^{\lambda})$ in the $x_1$ direction. Applying the moving plane method to all directions we obtain the radial symmetry of $u$.

\medskip

\textbf{Step 2:}
Now we prove the uniqueness of the solution of (\ref{equ-liou-2}) modulus scaling in (\ref{rescale}).
Let $w_1$ and $w_2$ be two radial solutions of (\ref{equ-liou-2}) satisfying $w_1(0)=w_2(0)$. Our goal is $w_1\equiv w_2$. First we make a remark about the smoothness of $w_1-w_2$. Indeed, by Lemma \ref{lem-BM} we can find some large $\beta $ such that $e^{\beta w_i}$ is integrable, which together with $\gamma>-1$ implies that the right hand side is $L^{p_1}$ for some $p_1=1+\varepsilon_1>1$. Therefore, $\Delta w_i\in L^{q_1}$ for $q_1=\frac{2p_1}{2-p_1}>2$. By Sobolev embedding theorems we obtain the regularity of $w_i$: $w_i\in C^{0,2-\frac{4}{q_1}}$ if $q_1<4$, and $w_i\in C^{1,1-\frac{4}{q_1}}$ if $q_1\geq4$. Thus if we denote $\alpha_1=2-\frac{4}{q_1}=4-\frac{4}{p_1}$, we have
\begin{equation}\label{w1-w2}
    (w_1-w_2)(x)=O(|x|^{\alpha_1}).
\end{equation}
The equation of $w_1-w_2$ reads
$$\Delta^2(w_1-w_2)=6|x|^{4\gamma}(e^{4w_1}-e^{4w_2})=24|x|^{4\gamma}e^{4\xi}(w_1-w_2) $$
where $\xi$ comes from the mean value theorem. If $\gamma>\frac{1}{p_1}-1$, which means $4\gamma+\alpha_1>0$, we can choose some $q>2$ such that the right hand side is $L^q$. By elliptic estimates, we obtain that $w_1-w_2\in W^{4,q}\subset C^2$. Otherwise when $-1<\gamma\leq\frac{1}{p_1}-1$, by using the same method, we know the right hand side of this equation is $L^{p_2}$ integrable near the origin for some $p_2$ in 
\begin{equation}\label{p1-p2}
    p_1<p_2<\frac{1}{\frac{1}{p_1}-(1+\gamma)}, \quad \mbox{if }\quad \frac{1}{p_1}>1+\gamma. 
\end{equation}
\eqref{p1-p2} leads to a better regularity of $(w_1-w_2)(x)=O(|x|^{\alpha_2})$ for some $\alpha_2=4-\frac{4}{p_2}>\alpha_1$. Obviously this boot-strap argument leads to $\{p_k\}$ and $\{\alpha_k\}$ such that if $|x|^{4\gamma}e^{4\xi}(w_1-w_2)\in L^{p_k}$ then $(w_1-w_2)(x)=O(|x|^{\alpha_k})$. Moreover, $p_{k+1}>p_k$ and $\alpha_k=4-\frac{4}{p_k}$. For given $\gamma>-1$, we can obtain some $p_{k_0}$ such that $1+\gamma-\frac{1}{p_{k_0}}>0$ after finite steps. Therefore, we can find some $q>2$ such that the right hand side is $L^q$, which means $w_1-w_2\in W^{4,q}\subset C^2$ as well. Note that since $w_i$ are radial, we certainly have $(w_1-w_2)'(0)=0$.  By the uniqueness of ODE, we only need to prove $(w_1-w_2)^{''}(0)=0$.

If $(w_1-w_2)^{''}(0)<0$, $w_1(r)<w_2(r)$ for small $r>0$. We will prove $w_1(r)<w_2(r)$ for all $r>0$. Suppose there exists $r_0>0$ such that $w_1(r_0)=w_2(r_0)$ and $w_1(r)<w_2(r)$ for $0<r<r_0$. Then by (\ref{equ-liou-2}), we have
\begin{equation*}
	\frac{\partial}{\partial r}\Delta (w_1(r)-w_2(r))=6r^{4\gamma}\big(e^{4w_1(r)}-e^{4w_2(r)}\big)<0,\quad 0<r\leq r_0,
\end{equation*}
which together with the assumption implies
\begin{equation*}
	\Delta (w_1-w_2)<0,\quad {\rm in} \ \, B(0,r_0).
\end{equation*}
Since  $w_1(r_0)-w_2(r_0)=0$, from the maximum principle, we have $w_1(r)-w_2(r)>0$ for $0<r< r_0$, which contradicts with $w_1(0)=w_2(0)$. Thus, $w_1(r)<w_2(r)$ for all $r>0$. Hence $\frac{\partial}{\partial r}\Delta (w_1(r)-w_2(r))<0$ for all $r>0$, which means $\Delta (w_1(r)-w_2(r))$ is decreasing in $r$. Thus $w_1(r)-w_2(r)\leq -c r^2$ as $r\to +\infty$ for some constant $c>0$, which yields a contradiction to the assumption $w_i(r)=o(r^2)$ at $\infty$.

Similarly, it is impossible for $(w_1-w_2)^{''}(0)>0$. Thus, the radial solution of (\ref{equ-liou-2}) is unique under the scaling $u_{\lambda}(x)=u(\lambda x)+(1+\gamma)\log \lambda$ for some $\lambda>0$, and it is valid for (\ref{equ-liou}) after scaling.

\end{proof}

\medskip

\noindent{\bf Proof of Corollary \ref{precise-u}:}  First we use the rough expansion of $u$ in (\ref{u-asy}) to rewrite $|x|^{4\gamma}e^{4u}$ as
$$r^{4\gamma}e^{4u}=e^{4c_0}r^{-4-4\mu}+O(r^{-4-4\mu-\tau}), \quad r>1, \quad \mbox{for some} \quad \tau>0. $$
Here we use $\mu=1+\gamma>0$ for convenience.
Then for $v=-\Delta u$, we have (see (\ref{lap-u-asy}))
$$\lim_{r\to \infty}v'(r)r^3=-8\mu.$$
  The equation for $v$ can be written as
$$v''(r)+\frac{3}{r}v'(r)=6e^{4c_0}r^{-4-4\mu}+O(r^{-4-4\mu-\tau}),\quad r>1. $$
Multiplying $r^3$ to both sides and integrating from $r$ to $\infty$, we have
$$r^3v'(r)+8\mu=-\frac{6e^{4c_0}}{4\mu}r^{-4\mu}+O(r^{-4\mu-\tau}),\quad r>1. $$
Thus
$$v'(r)=-\frac{8\mu}{r^3}-\frac{3e^{4c_0}}{2\mu}r^{-3-4\mu}+O(r^{-3-4\mu-\tau}), \quad r>1 $$
and
$$v(r)=\frac{4\mu}{r^2}+\frac{3e^{4c_0}}{2\mu(2+4\mu)}r^{-2-4\mu}+O(r^{-2-4\mu-\tau}),\quad r>1, $$
where $\lim_{r\to \infty}v(r)=0$ is used. Multiplying the expression of $v$, which is $-u''-\frac{3}{r}u'$, we have
\begin{equation}\label{tem-cor}
(r^3u')'(r)=-4\mu r-\frac{3e^{4c_0}}{2\mu(2+4\mu)}r^{1-4\mu}+E, 
\end{equation}
where $E=O(r^{1-4\mu-\tau})$. Here we discuss under two cases, either $\gamma>-\frac{3}{4}$ (which is $\mu>\frac 14$) or $\mu\le \frac 14$. In the first case $M=[\frac{1}{4\mu}]=0$, so our goal is to prove 
\begin{equation}\label{u-expand-1}
u(r)=-2\mu\log r+c_0+O(|x|^{-1}),\quad |x|>1. 
\end{equation}
Integrating (\ref{tem-cor}) from $1$ to $r$, we have
$$u'(r)=\frac{c}{r^3}-\frac{2\mu}r+\frac 1{r^3}\int_1^rE(s)ds, $$
where $c$ is a constant. Integrating the above again from $1$ to $r$ we obtain (\ref{u-expand-1}). Now we consider the case  $\mu\le \frac 14$. In this case
we use the fact that $\tau=4\mu-\epsilon$. Integration from $1$ to $r$ we have

\begin{equation}\label{u-step-1}
u(r)=-2\mu\log r+c_0+c_1r^{-4\mu}+O(r^{-8\mu+\epsilon})+O(r^{-1}),\quad r>1
\end{equation}
where
$$c_1=\frac{3e^{4c_0}}{32\mu^2(1+2\mu)(1-2\mu)}.$$
Thus if $\mu>\frac 18$ the expansion has an error term $O(r^{-1})$ and is finished (here we also note that in this case $M=[\frac{1}{4\mu}]=1$). 
So we only need to consider the case $\mu\le \frac 18$. 
Now (\ref{u-step-1}) has improved the estimate of $r^{4\gamma}e^{4u}$ to
\begin{equation}\label{impro-eu}
r^{4\gamma}e^{4u}=e^{4c_0}r^{-4\mu-4}+e^{4(c_0+c_1)}r^{-4-8\mu}+O(r^{-4-12\mu+\epsilon}),\quad r>1.
\end{equation}
Using (\ref{impro-eu}) in computation we obtain
$$u(r)=-2\mu\log r+c_0+c_1r^{-4\mu}+c_2r^{-8\mu}+O(r^{-12\mu+\epsilon})+O(r^{-1}), \quad r>1 $$
where
$$c_2=\frac{3e^{4(c_0+c_1)}}{128 \mu^2(1+4\mu)(1-4\mu)}. $$
This expression further improves the estimate of $r^{4\gamma}e^{4u}$. 
In general
$$c_l=\frac{3e^{4(c_0+...+c_{l-1})}}{32 l^2 \mu^2(1-2l\mu)(1+2l \mu)}.  $$
Obviously this process can be finished in finite steps as $\mu>0$. 
Finally the expansions of $\Delta u$ as well as other derivatives can be obtained in standard argument. Corollary \ref{precise-u} is established. $\Box$

\medskip

Next, we consider the case without the assumption $|u(x)|=o(|x|^2)$ at $\infty$.

The following lemma is similar to Lemma 3.3 in \cite{lin-classification}.

\begin{lem}\label{lem-poly}
	Suppose that $\Delta u =a$ in $\mathbb{R}^n$ for a constant $a\in \mathbb{R}$ such that $\exp(u-c|x|^2)\in L^1(\mathbb{R}^n)$ for some $c>0$. Then $u$ is a polynomial of order at most 2.
\end{lem}

\begin{proof}[\textbf{Proof}]

Let $P$ be a parabola that satisfies $\Delta P=a$. Then $u-P$ is a harmonic function, which also satisfies $exp(u-P-c|x|^2)\in L^1(\mathbb R^n)$ for some $c>0$ obviously. Thus
Lemma 3.3 in \cite{lin-classification} asserts that $u-P$ is a parabola, so is $u$.
	\end{proof}

If the $o(|x|^2)$ assumption is removed, we have the following result:

\begin{thm}\label{thm-classification-2}
	Let $u$ be a solution of (\ref{equ-liou-2}) with $\gamma>-1$. Then after an orthogonal transformation, $u(x)$ can be represented by (\ref{rep-u-2}), which has an asymptotic expansion of
	$$
	u(x)=-2(1+\gamma)\log|x|-\sum_{j=1}^{4}a_j(x_j-x_j^0)^2+c_0+O(|x|^{-\tau}), \quad |x|>1  $$
	for some $\tau>0$. The function $\Delta u$ satisfies
	\begin{equation}\label{lap-u-2}
	\Delta u(x)=-\frac{3}{2\pi^2}\int_{\mathbb{R}^4}\frac{1}{|x-y|^2}|y|^{4\gamma}e^{4u(y)}{\rm d}y-2\sum_{j=1}^{4}a_j, \quad |x|>1
	\end{equation}
	where $a_j$ are nonnegative constants and $x^0=(x_1^0,\cdots,x_4^0)\in\mathbb{R}^4$.
	Moreover, if $-1<\gamma<0$, $a_i x_i^0=0$ for all $i=1,\cdots,4$ and $a_1=a_2=a_3=a_4$, $u$ is radially symmetric.
	\end{thm}

\begin{proof}[\textbf{Proof of Theorem \ref{thm-classification-2}}]
	
	  Suppose that $u$ is a solution of (\ref{equ-liou-2}). Let $v$ be defined as in (\ref{v-def}) and $w(x)=u(x)+v(x)$. By Lemma \ref{lem-lap-u}, we have $\Delta w(x)\equiv-C_1\leq 0$ in $\mathbb{R}^4$. Since $v$ has only logarithmic growth at infinity, the integrability of $u$ guarantees that the assumption of Lemma \ref{lem-poly} is satisfied. Thus there exist constants $c_0$ and $a_{ij}$ $(i,j=1,\cdots,4)$ such that $a_{ij}=a_{ji}$ and
	  \begin{equation*}
	  	w(x)=\sum_{i,j,k=1}^{4}(a_{ij}x_ix_j+b_kx_k)+c_0.
	  \end{equation*}
	  After an orthogonal transformation, we may assume
	  \begin{equation*}
	  	u(x)=\frac{3}{4\pi^2}\int_{\mathbb{R}^4}\log\big(\frac{|y|}{|x-y|}\big)|y|^{4\gamma}e^{4u(y)}{\rm d}y-\sum_{j=1}^{4}(a_jx_j^2+b_jx_j)+c_0.
	  \end{equation*}
	  Since $|x|^{4\gamma}e^{4u}\in L^1(\mathbb{R}^4)$, we have $a_j\geq 0$ for all $j=1,\cdots,4$, and $b_j=0$ if $a_j=0$. Hence, we can rewrite $u(x)$ as follows:
	  \begin{equation*}
	  	u(x)=\frac{3}{4\pi^2}\int_{\mathbb{R}^4}\log\big(\frac{|y|}{|x-y|}\big)|y|^{4\gamma}e^{4u(y)}{\rm d}y-\sum_{j=1}^{4}a_j(x_j-x_j^0)^2+c_0
	  \end{equation*}
	  and (\ref{lap-u-2}) holds from the previous argument.

	  \medskip
	  Next, we show the radial symmetry under the assumption $a_ix_i^0=0$ for all $i$ and $a_1=a_2=a_3=a_4$. Clearly in this case
\begin{equation*}
	  	u(x)=\frac{3}{4\pi^2}\int_{\mathbb{R}^4}\log\big(\frac{|y|}{|x-y|}\big)|y|^{4\gamma}e^{4u(y)}{\rm d}y-\sum_{j=1}^{4}a_jx_j^2+c_0
	  \end{equation*}

 Let $\hat{u}(x)=u(x)+\sum_{j=1}^{4}a_jx_j^2$. Then
	  \begin{equation}
	  \Delta^2\hat{u}(x)=6|x|^{4\gamma}e^{-4\sum_{j=1}^{4}a_jx_j^2}e^{4\hat{u}(x)},\quad{\rm in}\ \, \mathbb{R}^4.
	  \end{equation}
	  As in Lemma \ref{lem-u-asy}, we set $\hat{w}(x)=\hat{u}(\frac{x}{|x|^2})-\alpha\log|x|$, then $|\hat{w}(x)|=o(\log\frac{1}{|x|})$ near 0 from Lemma \ref{lem-v-upper} and Lemma \ref{lem-v-lower}, and $\hat{w}$ satisfies
	  \begin{equation}
	  \Delta^2\hat{w}(x)=6|x|^{4\gamma}e^{-4\sum_{j=1}^{4}a_j(\frac{x_j}{|x|^2})^2}e^{4\hat{w}(x)},\quad{\rm in}\ \, \mathbb{R}^4\setminus\{0\}.
	  \end{equation}
	  Note that $a_j\geq 0$, then we can follow the argument in the proof of Lemma \ref{lem-u-asy} to obtain for $|x|$ large and a $\tau\in (0,1)$ such that
	  \begin{equation}
	  \hat{u}(x)=-\alpha\log|x|+c_0+O(|x|^{-\tau}),
	  \end{equation}
	  and
	  \begin{equation}
	  \left\{\begin{array}{ll}
	  -\Delta \hat{u}(x)=\frac{2\alpha}{|x|^2}+O(|x|^{-2-\tau}),	\\
	  -\frac{\partial}{\partial x_i}\Delta \hat{u}(x)=-4\alpha\frac{x_i}{|x|^4}+O(|x|^{-3-\tau}).
	  \end{array}
	  \right.
	  \end{equation}
	  At this point, we establish (\ref{rep-u-2}). Note that for $\hat{w}_{\lambda}=\hat{u}-\hat{u}_{\lambda}$, we have the expression as in \eqref{w-more}
	  \begin{align*}
	      \hat{w}_{\lambda}(x)=&\frac 3{4\pi^2}\int_{\Sigma_{\lambda}}\log \frac{|x^{\lambda}-y|}{|x-y|}
          \bigg (|y|^{4\gamma}e^{-4\sum_{j=1}^{4}a_jy_j^2}(e^{4u}-e^{4u_{\lambda}}) \\
          &+(|y|^{4\gamma}e^{-4\sum_{j=1}^{4}a_jx_j^2}-|y^{\lambda}|^{4\gamma}e^{-4\sum_{j=1}^{4}a_j(y^{\lambda})_j^2})e^{4u}\bigg )dy.
	  \end{align*}
	  Moreover, for $-\Delta \hat{u}+\Delta \hat{u}_{\lambda}$, we can get similar expression as in \eqref{v-more}. 
	  Since $a_j\leq 0$ and $\gamma<0$, we have $|y|^{4\gamma}e^{-4\sum_{j=1}^{4}a_jy_j^2}\geq |y^{\lambda}|^{4\gamma}e^{-4\sum_{j=1}^{4}a_j(y^{\lambda})_j^2}$ in $\Sigma_{\lambda}$, where we use the fact $\lambda\leq0$. Consequently, we can apply the method of moving planes as in the proof of Theorem \ref{thm-classification} to show that $\hat{u}(x)$ is symmetric with respect to the origin.

\end{proof}

\section[preliminaries]{Preliminaries for blowup analysis}\label{preliminaries}

Let $G:M\times M\setminus {\rm diag}$ denote the Green's function for the Paneitz operator
\begin{equation}\label{Green-function}
f(x)-\bar{f}=\int_M G(x,y)P_g f(y){\rm d}V_g(y),\quad \int_M G(x,y){\rm d}V_g(y)=0,
\end{equation}
where $\bar{f}=\frac{1}{{\rm vol}_g(M)}\int_M f{\rm d }V_g$ is the average of $f$ over $M$. Then the weak form of (\ref{Green-function}) is
\begin{equation}\label{Green-weakequ}
P_{g,y}G(x,y)=\delta_x-\frac{1}{{\rm vol}_g(M)}.
\end{equation}

Set $R$ be the regular part of the Green function. Then by the Appendix A in \cite{zhang-weinstein}, for $y$ in a neighborhood of $x$,
\begin{equation}\label{Green-func-expression}
G(x,y)=-\frac{1}{8\pi^2}\log d_g(x,y) \chi+R(x,y).
\end{equation}
where $\chi$ is a cut-off function to avoid cut locus.
Using $G$ we can decompose $u_k$ as the sum of its regular part and singular part
\begin{equation}\label{decompose-uk}
u_k(x)=\tilde{u}_k(x)-8\pi^2\sum_{j=1}^{N}\gamma_j G(x,q_j).
\end{equation}
Then $\tilde{u}_k$ satisfies
\begin{equation}\label{equ-tilde-uk}
P_g\tilde{u}_k+2b_k=2H_ke^{4\tilde{u}_k}\quad {\rm in} \ \, M,
\end{equation}
where
\begin{equation}\label{Hk(x)}
H_k(x)=h_k(x)\prod_{j=1}^N e^{-32\pi^2\gamma_j G(x,q_j)}.
\end{equation}
Clearly, (\ref{Q-equation-blowup}) and (\ref{assumption-coe}) imply that
\begin{equation}\label{finite-integral-tildeuk}
\int_M H_ke^{4\tilde{u}_k}{\rm d}V_g\leq C
\end{equation}

We will work with $\tilde{u}_k$ in the later blow-up analysis.

Similar with \cite{zhang-weinstein}, since the metric $g$ may not be locally conformally flat, we will apply the {\itshape conformal normal coordinates}, whose existence has been proved in \cite{Lee-Parker}. More specially, for $q\in M$, there exists a normal coordinate around $q$ such that $g$ can be deformed to $\mathfrak{g}$ which satisfies $\det\,(\mathfrak{g})=1$. We use $R_{ijkl}$ to denote the curvature tensor under $\mathfrak{g}$.

We will apply the expansions of the metric $\mathfrak{g}$ and its derivatives in the conformal normal coordinates (seen in the Appendix B of \cite{zhang-weinstein}), which are
\begin{equation}\label{expansion-g}
	\begin{split}
		&\mathfrak{g}_{ab}(x)=\delta_{ab}+\frac{1}{3}R_{aijb}x^ix^j+O(r^3),\\
		&\mathfrak{g}^{ab}(x)=\delta_{ab}-\frac{1}{3}R_{aijb}x^ix^j+O(r^3),\\
		&\partial_c\mathfrak{g}^{ab}(x)=-\frac{1}{3}(R_{acib}+R_{aicb})x^i+O(r^2),  \\
		&\partial_{ab}\mathfrak{g}^{ab}(x)=\frac 13 R_{ia,a}x^i+O(r^2),
	\end{split}
\end{equation}

In addition, the following Pohozaev identity from the Appendix D in \cite{zhang-weinstein} will play an important role when the blow-up analysis is carried out in the conformal normal coordinates.
\begin{lemA}\cite{zhang-weinstein}
	
	For equation $P_gu+2b=2he^{4u}$ in $M$ and $\Omega=B(0,r)$, there holds
	\begin{equation}\label{PI-mfd}
		\begin{split}
			&\int_{\Omega}\Big(2he^{4u}+\frac{1}{2}x^i\partial_i he^{4u}\Big)  \\
			=&\int_{\partial\Omega}\Big(\frac{1}{2}x^i\nu_ihe^{4u}-x^k\nu_jg^{ij}\partial_i(\Delta_g u)\partial_k u +\nu_jg^{ij}\Delta_gu\partial_iu +x^k\nu_jg^{ij}\Delta_gu \partial_{ik} u-\frac{1}{2}x^i\nu_i(\Delta_gu)^2 \Big)  \\
			&+\int_{\Omega}\Big(\Delta_gu\partial_i g^{ij}\partial_ju+x^k\Delta_gu\partial_{ik} g^{ij}\partial_ju+x^k\Delta_gu\partial_kg^{ij}\partial_{ij}u-2bx^i\partial_iu\Big)  \\
			&+2\int_{\partial\Omega}\Big(R_{ij,l}(q)x^lx^k\nu_i\partial_ju\partial_ku +O(r^3)|\nabla u|^2\Big)  \\
			&-\int_{\Omega}\Big(2R_{ij,l}(q)\big(x^l\partial_j u\partial_i u+x^kx^l\partial_j u\partial_{ik}u\big)+O(r^2)|\nabla u|^2+O(r^4)|\nabla^2 u|\Big)
		\end{split}
	\end{equation}
	
\end{lemA}

Note that we use $B(p,r)$ to denote a ball centered at $p$ with radius $r$. Sometimes if the center is the origin, we use $B_r$ instead of $B(0,r)$.

\section[Blow-up analysis]{Blow-up analysis near the singularity}\label{blowup-local}

In this section, we focus on the blow-up analysis near $q_j$, and to simplify the notation, we will omit the subscript $j$. Similar to the argument in \cite{zhang-weinstein}, we will work in the conformal normal coordinates near $q$ from \cite{Lee-Parker}. To be specific, we can find some function $w$ defined on $M$, such that in a small neighborhood $B(q,\delta)$ of $q$, $\delta>0$, we have
\begin{equation}
	\det\,(\hat{g})=1
\end{equation}
in the normal coordinates of the conformal metric $\hat{g}=e^{2w}g$. For convenience we just use $g$ instead of $\hat g$. Note that in a neighborhood of $q$,
\begin{equation}\label{hat-w}
w(x)=O(d_g(x,q)^2).
\end{equation}
where $d_g(x,q)$ stands for the distance between $x$ and $q$ under metric $g$.

Using the conformal covariance property of $P_g$
 the function $\mathfrak{u}_k=\tilde{u}_k-w$ satisfies
\begin{equation}\label{equ-hat-uk}
	P_{g}\mathfrak{u}_k+2\mathfrak{b}_k=2H_ke^{4\mathfrak{u}_k},
\end{equation}
and
\begin{equation}\label{finite-int-hatuk}
	\int_MH_ke^{4\mathfrak{u}_k}{\rm d}V_{g}\leq C,
\end{equation}
where $2\mathfrak{b}_k=P_{g}w+2b_k$ and $\tilde{H}_k=H_ke^{4w}$. For simplity, we still denote $\tilde{H}_k$ by $H_k$.

We still use $G$ to denote the Green's function for $P_{g}$. Then we have the following Green's representation formula
\begin{equation}\label{Green-formula}
	\mathfrak{u}_k(x)=\bar{\mathfrak{u}}_k+2\int_MG(x,y)H_k(y)e^{4\mathfrak{u}_k(y)}{\rm d}V_{g}(y)-2\int_MG(x,y)\mathfrak{b}_k(y){\rm d}V_{g}(y),
\end{equation}
where $\bar{\mathfrak{u}}_k$ is the average of $\mathfrak{u}_k$ over $(M,g)$. Using the expression of $G$ in (\ref{Green-func-expression}), we have
\begin{equation}\label{Green-rep-formula}
	\mathfrak{u}_k(x)=\bar{\mathfrak{u}}_k+2\int_M \Big(-\frac{1}{8\pi^2}\log d_g(x,y)\chi\Big) H_k(y)e^{4\mathfrak{u}_k(y)}{\rm d}V_{g}(y)+\mathfrak{\phi}_k(x),
\end{equation}
where
\begin{equation}\label{hat-phik}
	\mathfrak{\phi}_k(x)=2\int_MR(x,y)H_k(y)e^{4\mathfrak{u}_k(y)}{\rm d}V_{g}(y)-2\int_MG(x,y)\mathfrak{b}_k(y){\rm d}V_{g}(y).
\end{equation}

Note that $\det\,(g)=1$ in $B(q,\delta)$, we have ${\rm d}V_{g}(y)={\rm d}y$ in $B(q,\delta)$. Taking the difference of (\ref{Green-rep-formula}) evaluated at $x$ and $q$, we get
\begin{equation}
\begin{split}
	&\mathfrak{u}_k(x)-\mathfrak{u}_k(q)\\
	=&\frac{1}{4\pi^2}\int_M\big(\chi(r_q)\log|y-q|-\chi(r_x)\log d_g(x,y)\big)H_k(y)e^{4\mathfrak{u}_k(y)}{\rm d}V_{g}(y)+\mathfrak{\phi}_k(x)-\mathfrak{\phi}_k(q),
\end{split}
\end{equation}
where $r_q=|y-q|$ and $r_x=d_g(x,y)$. Here since the coordinates are normal, we have $d_g(y,q)=|y-q|$.

Thanks to the cut-off function $\chi$, we can replace the integral over $M$ by an integral over $B(q,2\delta)$:
\begin{equation}\label{difference-Green-for}
\begin{split}
&\mathfrak{u}_k(x)-\mathfrak{u}_k(q)\\
=&\frac{1}{4\pi^2}\int_{B(q,4\delta)}\big(\chi(r_q)\log|y-q|-\chi(r_x)\log d_g(x,y)\big)H_k(y)e^{4\mathfrak{u}_k(y)}{\rm d}V_{g}(y)\\
&+\mathfrak{\phi}_k(x)-\mathfrak{\phi}_k(q),\quad x\in B(q,2\delta).
\end{split}
\end{equation}

\medskip

We next  give an upper bound of the mass near $q$ when $\mathfrak{u}_k$ cannot blow up at $q$. Before stating such a small energy lemma we point out that the function $H_k$ can be written in the  neighborhood of a singular source $q$ as
\begin{equation}\label{eq:mathcalh}
  H_k(x) = \mathfrak{h}_k(x)d_g(x,q)^{4\gamma},
\end{equation}
with $ \mathfrak{h}_k(q) \neq 0$. Using this notation, our result states as follows:
\begin{lem}\label{lem-small-mass-regular}
Let $q$ be a singular source with index $\gamma$.	If
		\begin{equation}\label{smallness-condition}
			\varlimsup_{k\to +\infty}\int_{B(q,2\delta)}2\mathfrak{h}_kd_g(x,q)^{4\gamma}e^{4\mathfrak{u}_k(x)}{\rm d}V_{g}<\min\{8\pi^2,8\pi^2(1+\gamma)\},
		\end{equation}
 $\mathfrak{h}_k$ is defined in \eqref{eq:mathcalh}.
		Then $\mathfrak{u}_k\leq C$ in $B(q,\delta)$.
\end{lem}

\begin{proof}
	
	Note that
	\begin{equation*}
		P_{g}\mathfrak{u}_k=2\mathfrak{h}_kd_g(x,q)^{4\gamma}e^{4\mathfrak{u}_k(x)}-2\mathfrak{b}_k.
	\end{equation*}
	We write $\mathfrak{u}_k=u_{1k}+u_{2k}$ where $u_{1k}$ is the solution of
	\begin{equation}\label{smallness-equ}
		\left\{\begin{array}{lcl}
			\Delta^2 u_{1k}=2\mathfrak{h}_kd_g(x,q)^{4\gamma}e^{4\mathfrak{u}_k(x)}&& {\rm in} \ \, B(q,2\delta)\\
			u_{1k}(x)=\Delta u_{1k}(x)=0 && {\rm on} \ \, \partial B(q,2\delta).
			\end{array}
		\right.
	\end{equation}
	By Lemma 2.3 of \cite{lin-classification} we have
	\begin{equation}\label{BM-thm-1}
	\int_{B(q,2\delta)}\exp\Big\{\frac{\tilde{\delta}|u_{1k}|}{\parallel 2H_ke^{4\mathfrak{u}_k} \parallel_{L^1(B(q,2\delta),g)}}\Big\}{\rm d}V_{g}\leq C,
	\end{equation}
	with any $\tilde{\delta}\in(0,32\pi^2)$ and  some constant $C=C(\tilde{\delta},\delta)$. On one hand in $B(q,2\delta)$,
	\begin{equation}\label{u1k}
		u_{1k}(x)=\int_{B(q,\delta)}G_{\delta}(x,y)2\mathfrak{h}_kd_g(\eta, q)^{4\gamma}e^{4\mathfrak{u}_k}{\rm d}\eta, \quad x\in B(q,2\delta)
	\end{equation}
	where $G_{\delta}(x,y)$ is the Green's function of $\Delta^2$ on $B(q,2\delta)$:
$$G_{\delta}(x,y)=-\frac{1}{8\pi^2}\log |x-y|+R_{\delta}(x,y), $$
with
$$ G_{\delta}(x,y)=\Delta_yG_{\delta}(x,y)=0,\quad \mbox{for } x\in B(q,2\delta), y\in \partial B(q,2\delta). $$
In particular for $x\in B(q, \frac 32\delta)$,
$$u_{1k}(x)=-\frac 1{8\pi^2}\int_{B(q,2\delta)}\log |x-q|2\mathfrak{h}_k|\eta -q|^{4\gamma}e^{4\mathfrak{u}_k}{\rm d}\eta+O(1),\quad x\in B(q,\frac 32\delta). $$
	On the other hand the Green's representation formula of $\mathfrak{u}_k$ gives
	$$\mathfrak{u}_k(x)=u_{1k}(x)+u_{2k}(x)=\overline{ \mathfrak{u}_k}+\int_MG(x,\eta)2\mathfrak{h}_kd_g(x,q)^{4\gamma}e^{4\mathfrak{u}_k}{\rm d}\eta. $$
	Since the leading term of $G$ and $G_{\delta}$ are both $-\frac 1{8\pi^2}\log d_g(x,q)$, we have
	$$u_{2k}(x)=\overline{\mathfrak{u}_k}+O(1). $$
	From $\int_Me^{4\mathfrak{u}_k}{\rm d}V_g\le C$ and Jensen's inequality
	$$e^{4\bar{\mathfrak{u}}_k}\le \int_M e^{4\mathfrak{u}_k}{\rm d}V_g\le C.$$
	Therefore, $u_{2k}\leq C$ in $B(q,\frac{3}{2}\delta)$. Now we focus on $u_{1k}$.
	
	If $\gamma\ge 0$, (\ref{smallness-condition}) is
	\begin{equation*}
		\int_{B(q,2\delta)}2\mathfrak{h}_kd_g(x,q)^{4\gamma}e^{4\mathfrak{u}_k(x)}{\rm d}V_{g}<8\pi^2.
	\end{equation*}
Since $u_{2k}$ is bounded from above in $B(q,\frac 32\delta)$, we see from
 (\ref{BM-thm-1}) that there exists some $p>1$ to make $e^{4u_{1k}}\in L^{p}(B(q,2\delta))$:
	\begin{equation}\label{Lp'-estimate}
		\parallel 2\mathfrak{h}_kd_g(x,q)^{4\gamma}e^{4\mathfrak{u}_k(x)}\parallel_{ L^{p}(B(q,\frac{3}{2}\delta))}\leq C.
	\end{equation}
The estimate (\ref{Lp'-estimate}) leads to a $L^{\infty}$ bound of $u_{1k}$ in $B(q,\delta)$ based on two reasons. First
	the integration of (\ref{u1k}) gives a $L^1$ bound of $u_{1k}$ in $B(q,\frac 32\delta)$:
	\begin{equation}
		\parallel u_{1k} \parallel_{L^1(B(q,\frac{3}{2}\delta))}\leq C.
	\end{equation}
	Second, the standard interior regularity results in \cite{Browder-regularity} (Theorem 1 in Section 3 of \cite{Browder-regularity}) gives
	\begin{equation*}
		\parallel u_{1k}\parallel_{W^{4,p}B(q,\delta)}\leq \parallel 2\mathfrak{h}_kd_g(x,q)^{4\gamma}e^{4\mathfrak{u}_k(x)}\parallel_{ L^{p}(B(q,\frac{3}{2}\delta))}+\parallel u_{1k} \parallel_{L^1(B(q,\frac{3}{2}\delta))}\leq C.
	\end{equation*}
	Thus we have obtained the $L^{\infty}$ bound of $u_{1k}$ in $B(q,\delta)$ by standard Sobolev embedding theorem.
	
	If $-1<\gamma<0$, (\ref{smallness-condition}) is
	\begin{equation*}
		\int_{B(q,2\delta)}2\mathfrak{h}_kd_g(x,q)^{-4|\gamma |}e^{4\mathfrak{u}_k(x)}{\rm d}V_{g}<8\pi^2(1-|\gamma |).
	\end{equation*}
Since $\frac{\tilde \delta}{8\pi^2(1-|\gamma |)}<\frac{4}{1-|\gamma|}$, this strict inequality makes it possible to choose  $\tilde{\delta}$ in (\ref{BM-thm-1}) close to $32\pi^2$ such that there is a $p>1$, 
$p<\frac{1}{|\gamma |}$ and $p<\frac{\tilde{\delta}}{32\pi^2(1-|\gamma |)}$.
	Then  H${\rm \ddot{o}}$lder inequality tells us that (\ref{Lp'-estimate}) is also true in this case and the $L^{\infty}$ bound of $u_{1k}$ over $B(q,\delta)$ follows immediately. The combination of the $L^{\infty}$ bound of $u_{1k}$ and the upper bound of $u_{2k}$ implies that $\mathfrak{u}_k\le C$ in $B(q, \delta)$.

\end{proof}

An immediate consequence of Lemma \ref{lem-small-mass-regular} is that blowup sequence only converges to point measures. The following theorem takes one step further to assert that the point measure is quantized and the bubbling solutions tend to $-\infty$ away from blowup points.

\begin{thm}[Concentration and Quantization]\label{thm-con-quan}
	
	Let $\{\mathfrak{u}_k\}$ be a sequence of solution to (\ref{equ-hat-uk}) with (\ref{finite-int-hatuk}). Assume that $q$ is the only blow-up point of $\mathfrak{u}_k$ in $B(q,2\delta)$, then as $k\to+\infty$, along a subsequence, there hold
	\begin{equation}\label{con-neg-infty}
		\mathfrak{u}_k\to-\infty,\quad uniformly\ \, on \ \, any\ \, compact \ \, set\ \, of\ \, B(q,\delta)\setminus\{q\},
	\end{equation}
	\begin{equation}\label{con-measure}
	2 \mathfrak{h}_kd_g(x,q)^{4\gamma}e^{4\mathfrak{u}_k}\to\beta\delta_q,\quad in \ \, the\ \, measure\ \, on\ \,B(q,\delta),
	\end{equation}
	with $\beta=16\pi^2(1+\gamma)$ and $\mathfrak{h}_k$ is defined in \eqref{eq:mathcalh}.
	
In particular, along a subsequence,
	\begin{equation}\label{con-mass}
		2 \int_{B(q,\delta)}H_k(x)e^{4\mathfrak{u}_k(x)}{\rm d}V_{g}\to16\pi^2(1+\gamma),\quad as\ \, k\to+\infty.
	\end{equation}
	
\end{thm}

\begin{proof}[\textbf{Proof}]
	
	Suppose $q$ is the only blowup pint in $B(q,2\delta)$, then we observe that for any given $K\subset\subset B(q,\delta)\setminus \{q\}$
\begin{equation}\label{bd-osi}
|\mathfrak{u}_k(x)-\mathfrak{u}_k(y)|\le C(K),\quad x,y\in K,
\end{equation}
because the $|x-q|$ is comparable to $|y-q|$ and the total integration of $\mathfrak{h}_ke^{4\mathfrak{u}_k}$ is bounded. It is easy to obtain (\ref{bd-osi}) from the Green's representation formula. Next we claim that 
\begin{equation}\label{uk-grad}
	|\nabla_{g}^j\mathfrak{u}_k(x)|_{g}\leq C(K),\quad {\rm in}\ \,K,\quad j=1,2,3
\end{equation}
Set $r_0=\frac{1}{4}dist(K,B(q,\delta)\setminus\{q\})$. The Green's representation formula implies
\begin{equation*}
	\nabla_{g}^j\mathfrak{u}_k(x)=-\frac{1}{4\pi^2}\int_{B(q,2\delta)}\nabla_{g,x}^j\log d_g(x,y)H_k(y)e^{4\mathfrak{u}_k(y)}{\rm d}V_{g}(y)+\nabla_{g}^j\phi_k(x).
\end{equation*}
We only need to show that
\begin{equation*}
	\int_{B(q,2\delta)}\big|\nabla_{g,x}^j\log d_{g}(x,y)\big|H_k(y)e^{4\mathfrak{u}_k(y)}{\rm d}V_{g}(y)\leq C(K).
\end{equation*}
By means of \eqref{rough-comparsion-dist} and the boundness of $\|H_ke^{4\mathfrak{u}_k}\|_{L^1(M)}$, we obtain
\begin{equation*}
	\begin{split}
		&\int_{B(q,2\delta)}\big|\nabla_{g,x}^j\log d_{g}(x,y)\big|H_k(y)e^{4\mathfrak{u}_k(y)}{\rm d}V_{g}(y) \leq \int_{B(q,2\delta)}\frac{1}{|x-y|^{j}}H_k(y)e^{4\mathfrak{u}_k(y)}{\rm d}V_g(y)  \\
		\leq&C\int_{B(q,2\delta)\cap B(x,r_0)}\frac{1}{|x-y|^{j}}H_k(y)e^{4\mathfrak{u}_k(y)}{\rm d}V_g+\int_{B(q,2\delta)\setminus B(x,r_0)}\frac{1}{|x-y|^{j}}H_k(y)e^{4\mathfrak{u}_k(y)}{\rm d}V_g\\
		\leq& C(K)r_0^{4-j}+Cr_0^{-j} \\
		\leq& C(K).
	\end{split}
\end{equation*}

\medskip

Then the equation for $\mathfrak{u}_k$ further provides the estimate for the fourth order derivatives of $\mathfrak{u}_k$:
$$|\nabla^{\alpha}\mathfrak{u}_k(x)|\le C(K), \quad x\in  K, \quad |\alpha|=4. $$
	
Now we prove (\ref{con-neg-infty}) by contradiction. Suppose that there exists a point $x_0\in B(q,\delta)\setminus\{q\}$ such that $\{\mathfrak{u}_k(x_0)\}_{k\in\mathbb{N}}$ is bounded from below. By (\ref{bd-osi}) we see that $\mathfrak{u}_k$ is bounded in $L^{\infty}$ norm in any compact subset of $B(q,\delta)\setminus\{q\}$. This fact and the gradients estimates of $\mathfrak{u}_k$ guarantee that along a subsequence
	\begin{equation*}
		\mathfrak{u}_k\to \mathfrak{u}_0\quad {\rm in}\ \,C_{loc}^{3,\sigma}(B(q,2\delta)\setminus\{q\}),
	\end{equation*}
	with some constant $\sigma\in(0,1)$ and the limit function $\mathfrak{u}_0$ solves
	\begin{equation*}
		P_{g}\mathfrak{u}_0(x)+2\mathfrak{b}_0(x)=2\mathfrak{h}_0(x)d_g(x,q)^{4\gamma}e^{4\mathfrak{u}_0(x)}  \quad {\rm in} \ \,B(q,2\delta)\setminus\{q\}.
	\end{equation*}
Around $q$ we use $\beta(r)$ and its limit to describe the concentration of energy:	
	\begin{equation*}
		\begin{split}
			&\beta_k(r)=\int_{B(q,r)}2\mathfrak{h}_k(x)d_g(x,q)^{4\gamma}e^{4\mathfrak{u}_k(x)}{\rm d}V_{g}   \\
			&\beta(r)=\lim_{k\to+\infty}\beta_k(r),\quad \beta=\lim_{r\to 0}\beta(r).
		\end{split}
	\end{equation*}
From Lemma \ref{lem-small-mass-regular} we see that if $-1<\gamma< 0$, $\beta\ge8\pi^2(1+\gamma)$; and if $\gamma\ge0$, $\beta\ge 8\pi^2$.
	Fixing any $r>0$ small, we integrate the equation of $\mathfrak{u}_k$ in $B(q,r)$ to obtain
	\begin{equation}\label{beta-1}
		\int_{B(q,r)}\big(P_{g}\mathfrak{u}_k+2\mathfrak{b}_k\big){\rm d}V_{g}=\int_{B(q,r)}H_ke^{4\mathfrak{u}_k}{\rm d}V_{g}=\beta_k(r).
	\end{equation}
	By means of the properties of the metric in the conformal normal coordinates (see Appendix C in \cite{zhang-weinstein}), we can rewrite the first term in the left hand of \eqref{beta-1} as:
	\begin{align}
		\int_{B(q,r)}P_{g}\mathfrak{u}_k{\rm d}V_{g}=&\int_{B(q,r)}\Big(\Delta_{g}^2\mathfrak{u}_k+{\rm div}_g\Big(\big(\frac{2}{3}R_g g-2{\rm Ric}_g\big)\nabla \mathfrak{u}_k\Big)\Big){\rm d}V_{g}   \nonumber \\
		=&\int_{B(q,r)}\partial_l\Big(g^{il}\partial_i(\Delta_{g}\mathfrak{u}_k)+\big(\frac{2}{3}R g_{ij}-2R_{ij}\big)\partial_m\mathfrak{u}_kg^{li}g^{mj}\Big)   \label{beta-2}\\
		=&\int_{\partial B(q,r)}\Big(g^{il}\partial_i(\Delta_{g}\mathfrak{u}_k)+\big(\frac{2}{3}R g_{ij}-2R_{ij}\big)\partial_m\mathfrak{u}_kg^{li}g^{mj}\Big)\nu_l,\nonumber
	\end{align}
	where $\nu$ denotes the unit outward normal to $\partial B(q,r)$. By letting $k\to +\infty$ and then $r\to 0$, \eqref{beta-1}$\sim$\eqref{beta-2} and the definition of $\beta$ imply
	\begin{equation*}
		\lim_{r\to 0}\int_{B(q,r)}\big(P_{g}\mathfrak{u}_0+2\mathfrak{b}_0\big){\rm d}V_{g}=\beta.
	\end{equation*}
	Therefore, $\mathfrak{u}_0$ satisfies, in the distribution sense,
	\begin{equation*}
		P_{g}\mathfrak{u}_0(x)+2\hat{b}_0(x)=2\mathfrak{h}_0(x)d_g(x,q)^{4\gamma}e^{4\mathfrak{u}_0(x)}+\beta \delta_q  \quad {\rm in} \ \,B(q,2\delta).
	\end{equation*}
	Using the Green's representation formula for $\mathfrak{u}_0$, we have
	\begin{equation}\label{hatu0-decompose}
		\mathfrak{u}_0(x)=-\frac{\beta}{8\pi^2}\log d_g(x,q)+v(x)+w(x)
	\end{equation}
	where the first term comes from the convolution of $-\frac 1{8\pi^2}\log d_g(x,y)\chi$ with $\beta \delta_q$, the second term $v$ comes from
the convolution of $-\frac 1{8\pi^2}\log d_g(x,y)\chi $ with $2\mathfrak{h}_0d_g(x,q)^{4\gamma}e^{4\mathfrak{u}_0}$:
	\begin{equation}\label{hatu0-v}
		v(x)=-\frac{1}{4\pi^2}\int_{B(q,\delta)}\log d_g(x,y)\mathfrak{h}_0(y)d_g(y,q)^{4\gamma}e^{4\mathfrak{u}_0(y)}{\rm d}V_{g}(y),
	\end{equation}
	and $w$ is the collection of insignificant other terms:
	\begin{equation}\label{hatu0-w}
		w\in C^4(B(q,2\delta)).
	\end{equation}
	For $v$ we use (\ref{hatu0-v}) to denote $v(x)$ in $B(q,\delta)$ and we extend it smoothly such that $v\equiv 0$ on $M\setminus B(q,2\delta)$.
Based on the definition of $v$ we now show that $v\in L^{\infty}(B(q,\delta))$ for all $\gamma>-1$. In fact, from (\ref{hatu0-v}) we have this lower bound of $v$ in $B(q,\delta)$:
	\begin{equation}\label{v-lb}
		v(x)\geq \frac{1}{4\pi^2}\log\frac{1}{\delta}\parallel V\parallel_{L^1(B(q,\delta))}\geq C\quad {\rm in} \ \, B(q,\delta)
	\end{equation}
	where
$$V(x)=\mathfrak{h}_0(x)d_g(x,q)^{4\gamma}e^{4\mathfrak{u}_0(x)}. $$
The lower bounds of $\mathfrak{h}_0$ and $v$ lead to a lower bound for $V(x)$:
	\begin{equation*}
		V(x)=\mathfrak{h}_0(x)d_g(x,q)^{4\gamma}e^{4\mathfrak{u}_0(x)}\geq Cd_g(x,q)^{4\gamma-\frac{\beta}{2\pi^2}}e^{4v(x)+4w(x)}\geq\frac{c}{d_g(x,q)^s}
	\end{equation*}
	with $s=\frac{\beta}{2\pi^2}-4\gamma$ and suitable $c>0$. Since $\|V\|_{L^1 (B(q,\delta))}<\infty$ we see immediately that $s<4$, which is
	\begin{equation}\label{beta-con-1}
		\beta<8\pi^2(1+\gamma).
	\end{equation}
	Thus there is no way for $\mathfrak{u}_k$ to be bounded from below away from singular source unless $\gamma>0$. We have proved (\ref{con-neg-infty}) for $\gamma\le 0$.
	For $\gamma>0$ we have an upper bound for $V(x)$:
	\begin{equation}
		V(x)\leq\frac{c}{d_g(x,q)^s}e^{4v(x)}\quad {\rm in} \ \, B(q,\delta), \quad \mbox{if}\quad \gamma>0.
	\end{equation}
	To proceed with the proof of $v\in L^{\infty}$ we observe from (\ref{hatu0-v}) and direct computation that
	\begin{equation*}
		\left\{\begin{array}{lcl}
		\Delta^2 v(x)=V(x)+\eta(x)&& {\rm in} \ \, B(q,2\delta)\\
		v(x)=\Delta v(x)=0 && {\rm on} \ \, \partial B(q,2\delta),
		\end{array}
		\right.
	\end{equation*}
	where $\eta$ is smooth in $B(q,2\delta)$. Note that the boundary condition of $v$ on $\partial B(q,2\delta)$ is based on the smooth extension of $v$ mentioned before. Now we employ a standard argument of Brezis-Merle \cite{BM} to obtain $e^{\kappa |v|}\in L^1(B(q,2\delta),g)$ for any constant $\kappa>0$. Indeed, let $0<\epsilon<1/\kappa$ and $V=V_1+V_2$ with $\parallel V_1\parallel_{L^1(B(q,2\delta))}<\epsilon$ and $V_2\in L^{\infty}(B(q,2\delta))$. Correspondingly we write  $v=v_1+v_2$, where $v_1$ solves
	\begin{equation*}
		\left\{\begin{array}{lcl}
		\Delta^2 v_1(x)=V_1(x)&& {\rm in} \ \, B(q,2\delta)\\
		v_1(x)=\Delta v_1(x)=0 && {\rm on} \ \, \partial B(q,2\delta),
		\end{array}
		\right.
	\end{equation*}
	and $v_2$ solves
	\begin{equation*}
		\left\{\begin{array}{lcl}
		\Delta^2 v_2(x)=V_2(x)+\eta(x)&& {\rm in} \ \, B(q,2\delta)\\
		v_2(x)=\Delta v_2(x)=0 && {\rm on} \ \, \partial B(q,2\delta),
		\end{array}
		\right.
	\end{equation*}

Choosing $\tilde{\delta}=32\pi^2-1$ in (\ref{BM-thm-1}), we find
$$\int_{B(q,2\delta)}e^{\kappa |v_1|}\leq\int_{B(q,2\delta)}e^{\frac{|v_1|}{\parallel V_1\parallel_{L^1(B(q,2\delta))}}}\leq C$$
based on the smallness of $\epsilon$. Then by standard elliptic regularity theory, we have $v_2\in L^{\infty}(B(q,2\delta))$. Consequently, $e^{\kappa |v|}\in L^1(B(q,2\delta),g)$. Since $\kappa$ is large it is possible to use H${\rm \ddot{o}}$lder inequality to
 obtain $V\in L^{p^*}(B(q,2\delta),{g})$ for some $p\in(1,\frac{4}{s})$ if $s\geq 0$ and $p\in(1,+\infty)$ if $s\leq 0$. Here $p^*=\frac{p}{p-1}$ denotes the conjugate of $p$. Thus we have proved $v\leq C$ in $B(q,2\delta)$.
	
From the $L^{\infty}$ bound of $v$ we can use two positive constants $c_1$ and $c_2$ to bound $V$ from above and below
	\begin{equation}\label{range-V}
		\frac{c_1}{d_g(x,q)^s}\leq V(x)=\mathfrak{h}_0(x)d_g(x,q)^{4\gamma}e^{4\mathfrak{u}_0(x)}\leq \frac{c_2}{d_g(x,q)^s},\quad \mbox{if}\quad \gamma>0.
	\end{equation}

	Next, we aim to derive a contradiction by taking advantage of the Pohozaev identity (\ref{PI-mfd}) in Section \ref{preliminaries}.
	Set $h(x)=\mathfrak{h}_0(x)d_g(x,q)^{4\gamma}$ and $\Omega=B(q,r)$ in (\ref{PI-mfd}), then direct computation and (\ref{comparsion-dist}) give rise to
	\begin{equation*}
		\int_{\Omega}x^i\partial_i\big(\mathfrak{h}_0(x)d_g(x,q)^{4\gamma}\big)e^{4\mathfrak{u}_0}=\int_{B(q,r)}\big(x^i\partial_i \mathfrak{h}_0(x)d_g(x,q)^{4\gamma}e^{4\mathfrak{u}_0}+x^i\mathfrak{h}_0\partial_i d_g(x,q)^{4\gamma}e^{4\mathfrak{u}_0}\big)
	\end{equation*}
	and
	\begin{equation*}
		\begin{split}
			\partial_i d_g(x,q)^{4\gamma}=&4\gamma d_g(x,q)^{4\gamma-1}\partial_i d_g(x,q)=4\gamma d_g(x,q)^{4\gamma-1}\big(\partial_i|x-q|+O(r)\big)\\
			=&4\gamma d_g(x,q)^{4\gamma-2}\big((x-q)^i+	O(r^2)\big)
		\end{split}
	\end{equation*}
	Immediately, together with (\ref{comparsion-dist}) we get
	\begin{equation*}
		\begin{split}
			&\int_{\Omega}\frac{1}{2}(x-q)^i\partial_i\big(\mathfrak{h}_0(x)d_g(x,q)^{4\gamma}\big)e^{4\mathfrak{u}_0}\\
			=&\int_{B(q,r)}2\gamma  \mathfrak{h}_0(x)d_g(x,q)^{4\gamma}e^{4\mathfrak{u}_0}+\int_{B(q,r)}\frac{1}{2}\partial_{\nu} \mathfrak{h}_0 d_g(x,q)^{4\gamma+1}e^{4\mathfrak{u}_0}+o_r(1)
		\end{split}
	\end{equation*}
	Therefore, for $r\to 0$
	\begin{equation}\label{LHS-PI}
		{\rm (LHS)\ \,of\ \,}(\ref{PI-mfd})=2(1+\gamma)\int_{B(q,r)}\mathfrak{h}_0d_g(x,q)^{4\gamma}e^{4\mathfrak{u}_0}+O(r)=(1+\gamma)\beta+o_r(1),
	\end{equation}
	where $\lim_{r\to 0}o_r(1)=0$. Denote the four integrals on the right hand side of (\ref{PI-mfd}) by $I_1$, $I_2$, $I_3$ and $I_4$, respectively.
	Thanks to the expansions of $g$, which are in (\ref{expansion-g}), we obtain that
	\begin{equation*}
		|I_2|\leq C\int_{B(q,r)}\Big(r|\nabla^2\mathfrak{u}_0||\nabla\mathfrak{u}_0|+|x-q||\nabla^2\mathfrak{u}_0||\nabla\mathfrak{u}_0|+r|x-q||\nabla^2\mathfrak{u}_0|+|x-q||\nabla \mathfrak{u}_0|\Big),
	\end{equation*}
	\begin{equation*}
		|I_3|\leq C\int_{\partial B(q,r)}\Big(|x-q|^2|\nabla\mathfrak{u}_0|^2+O(r^3)|\nabla\mathfrak{u}_0|^2\Big),
	\end{equation*}
	\begin{equation*}
		|I_4|\leq C\int_{B(q,r)}\Big(|x-q||\nabla\mathfrak{u}_0|^2+|x-q|^2|\nabla^2\mathfrak{u}_0||\nabla\mathfrak{u}_0|+O(r^2)|\nabla\mathfrak{u}_0|^2+O(r^4)|\nabla^2\mathfrak{u}_0|\Big),
	\end{equation*}

	Next, we shall estimate $|\nabla^j\mathfrak{u}_0|$ in $B(q,r)$ for $j=1,2,3$. Recalling (\ref{hatu0-decompose})$\sim$(\ref{hatu0-w}), it is important to consider the $\nabla^j v$ in $B(q,r)$ for $j=1,2,3$.By means of the Green's representation formula, we observe that
	\begin{equation*}
		\begin{split}
			|\nabla^jv(x)|\leq  C\int_{B(q,2\delta)}\frac{1}{d_g(x,y)^j}V(y)+O(1)
		\end{split}
	\end{equation*}
	In order to estimate the integral in the inequality above, we decompose $B(q,2\delta)$ into two parts
	\begin{equation*}
		\Omega_1=B(q,2\delta)\cap\Big\{d_g(x,y)\leq\frac{d_g(x,q)}{2}\Big\},\quad \Omega_2=B(q,2\delta)\setminus \Omega_1.
	\end{equation*}
	In this estimate we use $V(y)=O(1)d_g(y,q)^{-s}$ in (\ref{range-V}). Hence
	\begin{equation}\label{inte-omega1}
		\int_{\Omega_1}\frac{V(y)}{d_g(x,y)^j}{\rm d}V_{g}(y)\leq \frac{C}{d_g(x,q)^s}\int_{B(x,\frac{d_g(x,q)}{2})}\frac{1}{d_g(x,y)^j}{\rm d}V_{g}(y)\leq C d_g(x,q)^{4-s-j}.
	\end{equation}
	Using (\ref{range-V}) again, we obtain that
	\begin{equation*}
		\begin{split}
			\tilde{I}:=\int_{\Omega_2}\frac{1}{d_g(x,y)^j}V(y){\rm d}V_{g}(y)\leq C\int_{\Omega_2}\frac{1}{d_g(x,y)^j}\frac{1}{d_g(y,q)^s}{\rm d}V_{g}(y).
		\end{split}
	\end{equation*}
	
	Fixing some $t$ as follows
	\begin{equation}\label{t-range}
		t\in  \left\{\begin{array}{lcl}
		(0,\frac{4}{s}),&& s>0,\\
		(1,+\infty),&& s\leq 0,
		\end{array}
		\right.
	\end{equation}
	we have $-ts>-4$. It follows from the H$\ddot{\rm o}$lder inequality that
	\begin{equation*}
		\begin{split}
			\tilde{I}\leq  C\Big(\int_{\Omega_2}\frac{1}{d_g(x,y)^{jt^*}}{\rm d}V_{g}(y)\Big)^{\frac{1}{t^*}}\Big(\int_{\Omega_1}\frac{1}{d_g(y,q)^{st}}{\rm d}V_{g}(y)\Big)^{\frac{1}{t}}
			\leq  C\Big(\int^{\tilde{c}}_{\frac{d_g(x,q)}{2}}\frac{1}{\rho^{jt^*-3}}{\rm d}\rho\Big)^{\frac{1}{t^*}}
		\end{split}
	\end{equation*}
	where $t^*=\frac{t}{t-1}$ denotes the conjugate of $t$ and $\tilde{c}$ is some positive constant. Then direct computation and the fact $-ts>-4$ imply that
	\begin{equation}\label{inte-omega2}
		\tilde{I}=\int_{\Omega_2}\frac{1}{d_g(x,y)^j}V(y){\rm d}V_{g}(y)\leq
		\left\{\begin{array}{lcl}
		C|\log d_g(x,q)|^{\frac{1}{t^*}},&& {\rm if} \ \,jt^*=4,\\
		Cd_g(x,q)^{\frac{4}{t^*}-j}+C,&& {\rm if} \ \,jt^*\neq 4.
		\end{array}
		\right.
	\end{equation}
	In view of (\ref{t-range}), we get that
	\begin{equation*}
		t^*\in  \left\{\begin{array}{lcl}
		(\frac{4}{4-s},+\infty),&& s>0,\\
		(1,+\infty),&& s\leq 0.
		\end{array}
		\right.
	\end{equation*}
	Hence, there holds $\frac{4}{t^*}-j<4-s-j$. Consequently, from (\ref{inte-omega1}) and (\ref{inte-omega2}) there exists some $\tau>0$ such that for any $r\in(0,\delta)$
	\begin{equation}\label{est-dv-inball}
		|\nabla^j v(x)|\leq Cd_g(x,q)^{\tau-j}+C,\quad j=1,2,3,\quad x\in B(q,r).
	\end{equation}
	In fact, we may choose $\tau\in(0,1)$ if $jt^*=4$, and otherwise $\tau=\frac{4}{t^*}$. At this point, we obtain that
	\begin{equation}
		|\nabla^j \mathfrak{u}_0(x)|\leq Cd_g(x,q)^{-j},\quad j=1,2,3\quad x\in B(q,r).
	\end{equation}
	Thus by virtue of (\ref{comparsion-dist}) and (\ref{est-dv-inball}), we may adjust $\tau>0$ such that on $\partial B(q,r)$
	\begin{equation*}
		\begin{split}
			&\frac{\partial\mathfrak{u}_0}{\partial r}=-\frac{\beta}{8\pi^2}\frac{1}{|x-q|}+O(r^\tau)\frac{1}{|x-q|}+O(1),  \\
			&\Delta\mathfrak{u}_0=-\frac{\beta}{4\pi^2}\frac{1}{|x-q|^2}+O(r^\tau)\frac{1}{|x-q|^2}+O(1),  \\
			&\frac{\partial}{\partial r}\Big(r\frac{\partial\mathfrak{u}_0}{\partial r}\Big)=O(r^{\tau-1}),  \\
			&\frac{\partial\Delta\mathfrak{u}_0}{\partial r}=\frac{\beta}{2\pi^2}\frac{1}{|x-q|^3}+O(r^\tau)\frac{1}{|x-q|^3}+O(1).
		\end{split}
	\end{equation*}

	\medskip
	
	Therefore, the estimates of $I_2$, $I_3$ and $I_4$ can be improved:
	\begin{equation*}
		\begin{split}
			|I_2|&\leq C\int_{B(q,r)}\big(rd^{-3}+d^{-2}+1\big){\rm d}V_{g}\leq Cr^2, \\
			|I_3|&\leq C\int_{\partial B(q,r)}r^2d_g(x,q)^{-2}{\rm d}V_{g}\leq Cr^3, \\
			|I_4|&\leq C\int_{B(q,r)}d_g(x,q)^{-1}{\rm d}V_{g}\leq Cr^3. \\
		\end{split}
	\end{equation*}
	Finally for  $I_1$ and we use the expansions of $g^{ij}$ to obtain
	\begin{equation*}
		\begin{split}
			I_1=&\int_{\partial B(q,r)}\big(-r\nu_i\partial_i(\Delta_{g}\mathfrak{u}_0)\partial_{\nu}\mathfrak{u}_0+\Delta_{g}\mathfrak{u}_0\partial_{\nu}\mathfrak{u}_0+(x-q)^k\nu_i\Delta_{g}\mathfrak{u}_0\partial_{ik}\mathfrak{u}_0-\frac{1}{2}r(\Delta_{g}\mathfrak{u}_0)^2\big)\\
			&+O(r)\\
			=&\int_{\partial B(q,r)}\Big(-r\partial_{\nu}(\Delta\mathfrak{u}_0)\partial_{\nu}\mathfrak{u}_0+\Delta\mathfrak{u}_0\partial_{\nu}<x-q,\nabla\mathfrak{u}_0>-\frac{1}{2}r(\Delta\mathfrak{u}_0)^2\Big)+o_r(1),
		\end{split}
	\end{equation*}
	where we have used $\det\,(g)=1$ in $B(q,\delta)$ and $\Delta_{g}u=\partial_i(g^{ij}\partial_j u)$. Consequently,
	\begin{equation}\label{RHS-PI}
		{\rm (RHS)\ \,of\ \,}(\ref{PI-mfd})=I_1+o_r(1)=\frac{\beta^2}{16\pi^2}+o_r(1).
	\end{equation}
	
	Combining (\ref{LHS-PI}) and (\ref{RHS-PI}), we derive that $\beta=16\pi^2(1+\gamma)$, which yields a contradiction to (\ref{beta-con-1}) in the case $\gamma>0$. Therefore $\mathfrak{u}_k\to-\infty$ uniformly on any compact subset of $B(q,2\delta)\setminus\{q\}$, $\mathfrak{h}_k(x)d_g(x,q)e^{4\mathfrak{u}_k}(x)\to 0$ uniformly on any compact subset of $B(q,2\delta)\setminus\{q\}$ and
	\begin{equation*}
		\mathfrak{h}_k(x)d_g(x,q)e^{4\mathfrak{u}_k}(x)\to \beta\delta_q \ \,{\rm\ \,in\ \,the\ \,measure\ \,on\ \,} B(q,\delta).
	\end{equation*}
	
	In the end, we show the quantization $\beta$ is exactly $16\pi^2(1+\gamma)$. To see this, set $c_k=\dashint_{\partial B(q,\delta)}\mathfrak{u}_k{\rm d}\sigma$ and $\check{u}_k(x)=\mathfrak{u}_k(x)-c_k$. Then, we have $c_k\to-\infty$ and $\check{u}_k\to\check{u}$ in $C^4(B(q,\delta))$ as $k\to+\infty$. Moreover, there exists a smooth function $\check{v}$ such that $\check{u}(x)=-\frac{\beta}{8\pi^2}\log d_g(x,q)+\check{v}$. Taking advantage of the Pohozaev identity as before, we obtain $\beta=16\pi^2(1+\gamma)$ and
	\begin{equation*}
		2\mathfrak{h}_k(x)d_g(x,q)^{4\gamma}e^{4\mathfrak{u}_k(x)}\to 16\pi^2(1+\gamma)\delta_q.
	\end{equation*}

\end{proof}

Based on Theorem \ref{thm-con-quan} and its proof, we immediately obtain the following corollaries.

\begin{cor}
	
	Suppose that $\mathfrak{u}_k$ satisfies the assumptions in Theorem \ref{thm-con-quan}, then along a subsequence, there holds
	\begin{equation*}
		\mathfrak{u}_k-c_k\to-2(1+\gamma)\log d_g(x,q)+\hat{v},\quad in \ \,C^4_{loc}(B(q,2\delta)\setminus\{q\}),
	\end{equation*}
	where $c_k=\dashint_{\partial B(q,\delta)}\mathfrak{u}_k{\rm d}\sigma_{g}\to-\infty$ and $\hat{v}$ is a smooth function in $B(q,2\delta)$.
	
\end{cor}

\begin{cor}
	
	Suppose that $\mathfrak{u}_k$ satisfies
	$$P_{g} \mathfrak{u}_k(x)+2\mathfrak{b}_k=2H_ke^{4\mathfrak{u}_k}\quad in \ \, M$$
	with $\int_{B(q_j,2\delta)}2H_ke^{4\mathfrak{u}_k}{\rm d}V_{g}\to \rho_j<16\pi^2(1+\gamma_j)$ for some $j\in\{1,\cdots,N\}$. Then $\{\mathfrak{u}_k\}$ is uniformly bounded from above on any subset of $B(q_j,2\delta)$. In particular, $\{\mathfrak{u}_k\}$ can not blow up in $B(q_j,2\delta)$.
	
\end{cor}

\section[A Priori Estimate]{Concentration-Compactness Result and A Priori Estimate}\label{CC-Apriori}

In this section, we aim to establish the concentration-compactness principle and a priori estimate based on the result in the section \ref{blowup-local}. Indeed, we will derive the following concentration-compactness type result for the regular part of $\{u_k\}$.

Set
\begin{equation*}
\rho_k=\int_M 2H_ke^{4\mathfrak{u}_k}{\rm d}V_{\mathfrak{u}_k}.
\end{equation*}

\begin{thm}[Concentration-Compactness]\label{thm-concentration-compactness}
	
	Let $\{\tilde{u}_k\}$ be a sequence of solution to (\ref{equ-tilde-uk}) and (\ref{finite-integral-tildeuk}) with $\rho_k\to\rho$. Then there exists a subsequence, still denoted $\{\tilde{u}_k\}$, for which one of the following alternative holds:
	\begin{itemize}
		\item[(i)]
		$\sup_{\varSigma}|\tilde{u}_k|\leq C_{\varSigma}$, for any $\varSigma\subset\subset M$.
		\item[(ii)]
		$\sup_{\varSigma}\tilde{u}_k\to-\infty$, for any $\varSigma\subset\subset M$.
		\item[(iii)]
		There exist a finite set $S=\{p^1,\cdots,p^m\}\subset M$ with $m\in \mathbb{N}$, and sequences of points $\{x_k^1\}_{k\in\mathbb{N}},\cdots,\{x_k^m\}_{k\in\mathbb{N}}\subset M$, such that for all $i=1,\cdots,m$
		\begin{equation*}
		x_k^i\to p^i,\quad \sup_{\varSigma}\tilde{u}_k\to-\infty \ for\ \, any\ \,\varSigma\subset M\setminus S
		\end{equation*}
		and
		\begin{equation*}
		2H_ke^{4\tilde{u}_k}\to \sum_{i=1}^m\beta_i\delta_{p^i} \ weakly\ \,in\ \,the\ \,sense\ \,of\ \,measures\ \,in\ \,M.
		\end{equation*}
		Futhermore, $\beta_i\in 16\pi^2\mathbb{N}$ if $p^i\notin\{q_1,\cdots,q_N\}$, and $\beta_i=16\pi^2(1+\gamma_j)$ if $p^i=q_j$ for some $j\in\{1,\cdots,N\}$.
	\end{itemize}
	
\end{thm}

\begin{proof}[\textbf{Proof of Theorem \ref{thm-concentration-compactness}}]
	
	We define $S$ to be the set of blow-up points of $\mathfrak{u}_k$ in $M$, that is,
	\begin{equation*}
		S=\{x\in M:\ \,\exists x_k\in M,\ \,{\rm s.t.}\ \,x_k\to x\ \, {\rm and}\ \,\tilde{u}_k(x_k)\to+\infty\ \,{\rm as}\ \,k\to+\infty\}.
	\end{equation*}
	
	We distinguish two cases.
	
	\textbf{Case 1:} $S\neq\emptyset$.
	
	For $p\in S$,  Lemma \ref{lem-small-mass-regular} say that the mass of $\{\tilde{u}_k\}$ near $p$ is no less than $8\pi^2(1+\gamma)$.
	Then finite integral assumption $\int_M H_ke^{4\tilde{u}_k}{\rm d}V_{g}\leq C$ implies ${\rm card}\,(S)\leq C$. We may denote $S=\{p_1,\cdots,p_m\}$ with some $m\in\mathbb{N}$. Therefore, there exists $r_0\in (0,1)$ such that for any $p^i\in S$, $p^i$ is the only blow-up point of $\tilde{u}_k$ in $B(p^i,r_0)$. Therefore, from the results in \cite{Druet-Robert} and Theorem \ref{thm-con-quan}, we obtain the alternative (iii).
	
	\textbf{Case 2:} $S=\emptyset$.
	
	In this case, we have $\sup_M\tilde{u}_k\leq C$, which implies $H_ke^{4\tilde{u}_k}$ is uniformly bounded in $M$. Taking into account of the Green's representation formula,
	\begin{equation*}
		\tilde{u}_k(x)-\bar{\tilde{u}}_k=\int_M G(x,y)H_k(y)e^{4\tilde{u}_k}{\rm d}V_{g}(y)=O(1).	
	\end{equation*}
	Hence, after taking a subsequence, the alternative (i) occurs if $\limsup_{k\to+\infty}\int_M\tilde{u}_k{\rm d}V_{g}>-\infty$,  the alternative (ii) holds if $\limsup_{k\to+\infty}\int_M\tilde{u}_k{\rm d}V_{g}\to-\infty$,
	
\end{proof}

Immediately, we derive the following two corollaries from Theorem \ref{thm-concentration-compactness}.

\begin{cor}
	
	Suppose that $\tilde{u}_k$ satisfies the assumption in Theorem \ref{thm-concentration-compactness} and alternative (iii) occurs, then $\rho\in\Gamma$.
	
\end{cor}

\begin{cor}
	
	Suppose that $\tilde{u}_k$ satisfies the assumption in Theorem \ref{thm-concentration-compactness}. Then for every $\rho\in\mathbb{R}^+\setminus\Gamma$, there exists a constant $C$ only depending on $\rho$, such that $\tilde{u}_k\leq C$. In particular, if we additionally assume $\int_M \tilde{u}_k{\rm d}V_g=0$ or $\int_M H_k e^{4\tilde{u}_k}{\rm d}V_g\geq c$ for some constant $c>0$, then
	\begin{equation*}
		\parallel\tilde{u}_k\parallel_{L^{\infty}(M)}\leq C.
	\end{equation*}
	
\end{cor}

The following result explains that $\Gamma$ is some critical set to (\ref{Q-equ-gen-singular}) or (\ref{Q-equation-blowup}).

\begin{prop}[Critical Set]\label{critical-set}
Suppose {\rm Ker}\,$(P_g)=\{constants\}$ and that $\{u_k\}$ is a sequence of solutions to (\ref{Q-equation-blowup})$\sim$(\ref{volume-normal}) with the coefficients satisfying (\ref{assumption-coe}). If the blow-up phenomena occur, then $\int_M2b{\rm d}V_g\in\Gamma$.
	
\end{prop}

\begin{proof}[\textbf{Proof of Propertion \ref{critical-set} and Theorem \ref{thm-apriori-est}}]
	
Since 
$$\int_MH_ke^{4\mathfrak{u}_k}{\rm d}V_{g}=\int_Mh_ke^{4u_k}{\rm d}V_g=\int_Mb_k{\rm d}V_g,$$ 
Proposition \ref{critical-set} and Theorem \ref{thm-apriori-est} obviously follow from the two corollaries above.
	
\end{proof}

Finally we prove Theorem \ref{sphe-har-uk}, which is obviously equivalent to the following form:
\begin{thm}[A Spherical Harnack Inequality ($\gamma\notin\mathbb{N}$)]\label{SHT}
	
	Suppose that $\tilde{u}_k$ satisfies the assumption in Theorem \ref{thm-concentration-compactness} and $q_j$ is a blow-up point of $\tilde{u}_k$. If $\gamma_j\notin \mathbb{N}$, then near $q_j$ there holds the following spherical Harnack inequality:
	\begin{equation}\label{spherical-Harnack}
		\max_{x\in B(q_j,\delta)}\{\tilde{u}_k(x)+(1+\gamma_j)\log|x-q|\}\leq C
	\end{equation}
	with some constant $C$.
	
\end{thm}

\begin{proof}[\textbf{Proof}]
	
	Suppose that (\ref{spherical-Harnack}) fails, then there exists a sequence $\{x_k\}\subset B(q_j,\delta)$ such that $\max_{x\in B(q_j,\delta)}\{\tilde{u}_k(x)+(1+\gamma_j)\log|x-q|\}\to +\infty$. Here we note that these finite points are chosen from a selection process \cite{Malchiodi,Druet-Robert}. We start from $q$, then $x_1^k$, $x_2^k$ and so on. Around each chosen point there is a bubbling ball in which a profile of global solution can be observed and a total integration of $16\pi^2+o(1)$ is inside the bubbling ball. These bubbling balls are finite due to the uniform bound on the total integration of bubbling solutions. If we enclose all these bubbling balls by a bigger ball, the calculation of Pohozaev identity over the bigger hall yields that the total integration of $H_k(x)e^{4\tilde u_k}$ is $16\pi^2(1+\gamma)$. From here we obtain the desired contradiction: First there is no bubbling ball centered around $q$, because otherwise the bubbling disk contributes $16\pi^2(1+\gamma)$, which is absurd since other bubbling balls contribute a multiple of $16\pi^2$, impossible to have the total equal to $16\pi^2(1+\gamma)$. Second, since $\gamma$ is not a positive integer, there is no way to have other bubbling balls, since each one of them contributes $16\pi^2$ in integration, there is no way for all of them to contrite $16\pi^2(1+\gamma)$ in total. 
	
\end{proof}

\section*{Appendix: Comparison between $d_g(x,q)$ and $|x-q|$}
\setcounter{equation}{0}
\setcounter{subsection}{0}
\renewcommand{\theequation}{A.\arabic{equation}}
\renewcommand{\thesubsection}{A.\arabic{subsection}}

In this appendix, we will establish the comparison between the distance $d_g(x,q)$ and its derivatives and their Euclidean counterparts as in the Appendix B in \cite{zhang-weinstein}. We will follow the argument in \cite{zhang-weinstein} and give the detail for completeness. We claim that for $j=0,1,2,3$, there holds
\begin{equation}\label{comparsion-dist}
	\nabla^j\big(\log|x-y|-\log d_g(x,y)\big)=O(r^{2-j}),\quad x\in B(q,2r)\setminus B(q,r/2).
\end{equation}

We recall that $g$ is the conformal normal metric centered at $q$, and we identify $x,y\in T_qM$ with $\exp_qx$ and $\exp_qy$ respectively, where $\exp_q$ is the exponential map at $q$ with respect to the metric $g$. Thus,
\begin{equation*}
	d_g(x,y)=d(\exp_qx,\exp_qy), \quad x,y\in B(q,\delta).
\end{equation*}
Also, we denote $\nabla_{g}$ by $\nabla$ for convenience. First, let us note that the following simple estimates on $d$ hold:
\begin{equation}\label{rough-comparsion-dist}
	\big|\nabla^j(\log d_g(x,y))\big|\leq C|x-y|^{-j},\quad j=1,2,3,4.
\end{equation}

Set
\begin{equation*}
	f(x)=\log|x-y|-\log d_g(x,y),\quad x\in B(q,2r)\setminus B(q,r/2).
\end{equation*}
We aim to show that
\begin{equation*}
	\big|\nabla^jf(x)\big|\leq C|x-q|^{2-j},\quad j=0,1,2,3,\quad x\in B(q,2r)\setminus B(q,r/2).
\end{equation*}

Let $R$, $R_{ij}$ and $R_{ijkl}$ respectively denote the scalar, Ricci and Riemann curvature of $g$. From the definitions of $g$ and $R^i_{jkl}$, we obtain that
\begin{equation*}
	\nabla^jR^i_{jkl}(x)=O(1),\quad j=1,2.
\end{equation*}
In conformal normal coordinates, there holds $R(q)=R_{ij}(q)=|\nabla R(q)|=0$. As a consequence, we have futher
\begin{equation*}
	R(x)=O(r^2),\quad R_{ij}(x)=O(r).
\end{equation*}

We shall derive an estimate on $\Delta_{g}^2f(x)$. By the definition of $g$ and (A.4) in \cite{zhang-weinstein}, that is
\begin{equation*}
	P_{g,y}\Big(-\frac{1}{8\pi^2}\chi(r)\log d_g(x,y)\Big)=\delta_x+E(x,y),\quad {\rm with \ \,}E \ \,{\rm bounded},
\end{equation*}
we have that
\begin{equation}
	P_{g}\log d_g(x,q)=O(r^4),\quad x\in B(q,2r)\setminus B(q,r/2).
\end{equation}
In view of the rough estimates (\ref{rough-comparsion-dist}), we can estimate the term:
\begin{equation*}
	\big(P_{g}-\Delta_{g}^2\big)\log d_g(x,q)=\partial_m\Big(g^{mi}\big(\frac{2}{3}R(x)g_{ij}-2R_{ij}(x)\big)g^{lj}\partial_j\big(\log d_g(x,q)\big)\Big)=O(r^{-1}).
\end{equation*}
Therefore, we can get
\begin{equation*}
	\Delta_{g}^2\big(\log d_g(x,q)\big)=O(r^{-1}),\quad x\in B(q,2r)\setminus B(q,r/2).
\end{equation*}
Finally, we consider the term $\Delta_{g}^2\big(\log |x-q|\big)$. Since $\Delta^2\big(\log |x-q|\big)=0$, it suffices to estimate $\Delta_{g}^2-\Delta^2$. For any function $u$, the direct computation leads to
\begin{equation}\label{expansion-4-order}
	\begin{split}
		\Delta_{g}^2u=&g^{ab}g^{ij}\partial_{ijab}u+2\partial_{ija}u\big(\partial_bg^{ab}g^{ij}+g^{ab}\partial_bg^{ij}\big) \\
		&+\partial_{ij}u\big(\partial_ag^{ab}\partial_bg^{ij}+2g^{ai}\partial_{ab}g^{bj}+g^{ab}\partial_{ab}g^{ij}+\partial_ag^{ia}\partial_bg^{bj}\big)  \\
		&+\partial_ju\big(\partial_ag^{ab}\partial_{ib}g^{ij}+g^{ab}\partial_{iab}g^{ij}\big),
	\end{split}
\end{equation}
where we have used $\det\,(g)=1$. Using the expansion of $g^{ab}$:
\begin{equation*}
	g^{ab}(x)=\delta_{ab}-\frac{1}{3}R_{mabl}(q)x^ax^b+O(|x-q|^3),
\end{equation*}
and replacing $u$ by $\log|x-q|$ in (\ref{expansion-4-order}) above, we obtain that
\begin{equation*}
	\Delta_{g}^2\big(\log|x-q|\big)=O(r^{-2}).
\end{equation*}
Consequently,
\begin{equation}\label{app-equ}
	\Delta_{g}^2 f(x)=O(r^{-2}),\quad  x\in B(q,2r)\setminus B(q,r/2).
\end{equation}

An estimate on the $L^{\infty}$-norm of $f(x)$ can easily be seen as follows:
\begin{equation}
	d_g(x,q)=\int_{0}^{x(t)}\sqrt{x_i'(t)g_{ij}(t)x_j'(t)}{\rm d}t=|x-q|\big(1+O(r^2)\big),
\end{equation}
which implies
\begin{equation}\label{app-C0}
	f(x)=O(r^2),\quad  x\in B(q,2r)\setminus B(q,r/2).
\end{equation}
Applying the elliptic theory to (\ref{app-equ}) and (\ref{app-C0}), we obtain the claim (\ref{comparsion-dist})
and
\begin{equation}\label{0-comparsion}
	d_g(x,q)=|x-q|(1+O(r^2)).
\end{equation}

{\color[rgb]{0.000000,0.000000,0.000000}Based}
\end{document}